\documentclass[reqno,11pt]{amsart}
\usepackage{amssymb,latexsym}
\usepackage{amssymb,amsmath,amsthm}
\usepackage{amsfonts}
\usepackage{latexsym}
\usepackage[CJKbookmarks=true]{hyperref} 
\usepackage{color}

\numberwithin{equation}{section}

\def\pa{\partial}

\newcommand{\R}{\mathbb{R}}

\newcommand{\norm}[1]{\left\| #1\right\|}

\newcommand{\Z}{\mathbb{Z}}

\newcommand{\les}{\lesssim}

\newcommand{\tc}{\tilde \chi}

\newcommand{\ppj}{(\phi_0^j,\phi^j_1)}

\newcommand{\BB}{\dot{B}^3_{1,1}\times\dot{B}^2_{1,1}}

\newcommand{\bb}{\dot{B}^{1/2}_{2,1}\times\dot{B}^{-1/2}_{2,1}}

\newcommand{\bt}{\dot{B}^{1/2}_{2,1}}

\newcommand{\uu}{(u,u_t)}

\newcommand{\RR}{(\tilde{R}_{0,n}^N, \tilde{R}_{1,n}^N )}

\newcommand{\uui}{(u_0,u_1)}
\newcommand{\tq}{\tilde{q}}
\newcommand{\tr}{\tilde{r}}

\newcommand{\tu}{\tilde{u}}

\newcommand{\tv}{\tilde{v}}

\newcommand{\tw}{\tilde{w}}

\newcommand{\bu}{\bar{u}}
\newcommand{\bv}{\bar{v}}
\newcommand{\bw}{\bar{w}}

\newcommand{\uuin}{(u_0^n,u_1^n)}

\newcommand{\HL}{\dot{H}^1\times L^2}
\newcommand{\vp}{\varphi}

\DeclareMathOperator*{\esssup}{ess\,sup}

\newcommand{\ra}{\rightarrow}

\newcommand{\ep}{\epsilon}

\newcommand{\vps}{\varepsilon}

-\marginparwidth \oddsidemargin -9mm \evensidemargin \oddsidemargin

\advance\hoffset by 5mm

\def\@abssec#1{\vspace{.05in}\footnotesize \parindent .2in
{\bf #1. }\ignorespaces}

\newcommand{\HH}{\dot{H}^{1/2}\times\dot{H}^{-1/2}}

\newtheorem{theorem}{Theorem}
\newtheorem{lemma}{Lemma}[section]
\newtheorem{proposition}[lemma]{Proposition}

\newtheorem{remark}[lemma]{Remark}

\theoremstyle{definition}%
\newtheorem{definition}{Definition}

\newcommand{\beq}{\begin{equation}}

\newcommand{\eeq}{\end{equation}}

\theoremstyle{remark}

\makeatletter
\newcommand{\Extend}[5]{\ext@arrow0099{\arrowfill@#1#2#3}{#4}{#5}}
\makeatother

\begin{document}
\title[$5$D NLW in Besov spaces]{The global well-posedness and scattering for the $5$D defocusing conformal invariant NLW  with radial initial data in a critical Besov space}

\author[Miao]{Changxing Miao}
\address{\hskip-1.15em Changxing Miao:
\hfill\newline Institute of Applied Physics and Computational
Mathematics, \hfill\newline P. O. Box 8009,\ Beijing,\ China,\
100088,} \email{miao\_changxing@iapcm.ac.cn}

\author[Yang]{Jianwei Yang}
\address{\hskip-1.15em Jianwei Yang:
 \hfill\newline Beijing International Center for Mathematical Research, Peking University, Beijing 100871, China }
\email{geewey\_{}young@pku.edu.cn}

\author[Zhao]{Tengfei Zhao}
\address{\hskip-1.15em Tengfei Zhao:
 \hfill\newline The Graduate School of China Academy of Engineering Physics, \hfill\newline P. O.
Box 2101,\ Beijing,\ China,\ 100088, } \email{zhao\_tengfei@yeah.net}

\subjclass[2000]{Primary: 35B40, 35Q40.}

\keywords{ Nonlinear wave equation, Strichartz estimates,  Scattering,   hyperbolic coordinates, Morawetz esimates}

\begin{abstract}

In this paper, we  obtain the global
well-posedness and scattering for the radial solution to the defocusing conformal invariant nonlinear
wave equation with initial data in the critical Besov space $\BB(\R^5)$.
This is the five dimensional analogue of \cite{dodson-2016}, which is
the first result on the global well-posedness and scattering of the
energy subcritical nonlinear wave equation without the uniform boundedness
assumption on the critical Sobolev norms employed as a substitute of the
missing conservation law with respect to the scaling invariance of the equation.
The proof is based on exploiting the structure of the radial solution,
developing the Strichartz-type estimates and incorporation of
the strategy in \cite{dodson-2016}, where we also avoid a logarithm-type loss
by employing the inhomogeneous Strichartz estimates.

\end{abstract}

 \maketitle

\section{Introduction}
\noindent We consider the solutions $u$ to 
\beq
\label{NLW-p}
\left\{
 \begin{aligned}
                 & \pa_{tt}u-\Delta u+\mu |u|^pu =\,0,\,\,\,\,\,\,\, (t,x)\in \R\times\R^d, \\
              & (u(0), \pa_t u(0)) =\, \uui,\;\;\;\;x\in \R^d,
  \end{aligned}\right.
\eeq
where $\mu=\pm1 $, $d\geq1$,  and $p>0$. 
If $\mu=1$ the equation \eqref{NLW-p} is described as defocusing,  otherwise focusing.
There is a natural scaling symmetry for \eqref{NLW-p}, i.e.,
if we let 
$
u_\lambda(t,x)=\lambda^{\frac{2}{p}}u(\lambda t,\lambda x)
$ for $\lambda>0$,
then $u_\lambda$ 
is also a solution to the equation \eqref{NLW-p} with initial
data $(\lambda^{\frac2p}u_0(\lambda x), \lambda^{\frac2p+1}u_1(\lambda x))$
preserving  the $\dot{H}^{s_p}\times \dot{H}^{s_p-1}(\R^d)$ norm of the initial data,
 %
where we define  the critial regulairity as $s_p=\frac d2-\frac{2}{p}$.
At least, the solutions to \eqref{NLW-p} formally  conserve the energy
\beq\label{Energy-p}
E(u(t), \pa_t u(t)):=\,  \frac12\int_{\R^5} |\nabla_x u(t) |^2 dx  +
 \frac12\int_{\R^5}|\pa_{t}u (t)|^2 dx +\frac{\mu}{p+2}\int_{\R^5}|u(t)|^{p+2} dx,
\eeq
which is also invariant under the scaling if $s_p=1$.
In view of this, we say the Cauchy problem  \eqref{NLW-p} is energy critical when $s_p=1$,
subcritical for $s_p<1$ and  supercritical when $s_p>1$.

Lindblad and Sogge \cite{lindblad1995on} proved the local theory of the Cauchy
problem \eqref{NLW-p} in the minimal regularity spaces. In fact, 
if $d\geq 2$ and $p\geq \frac{d+3}{d-1}$, the
Cauchy problem \eqref{NLW-p}
with initial data in the critical spaces
$\dot{H}^{s_p}\times \dot{H}^{s_p-1}(\R^d)$ is locally well-posed.
The global theory for the Cauchy problem \eqref{NLW-p} with $\mu=1$ and $s_p\leq 1$
has been studied extensively.
While for the focusing case, even the solution with smooth
initial data may  blow up at finite time. 
For more related results, we refer the reader
to \cite{sogge1995lectures} and the references therein.

In this paper, we will consider global existence and scattering
of the solutions to \eqref{NLW-p}.
In general, a solution $u$ is said to be scattering,
if it is a global solution and approaches a linear
solution 
as $t\ra\pm\infty$.
%
%
%
In the cases of  $d\geq 2$ and $p\geq \frac{d+3}{d-1}$, the solution to
\eqref{NLW-p}  with  small initial datum in the  critical Sobolev spaces
is globally well-posed and scattering, see \cite{lindblad1995on}. 

For the defocusing energy critical wave equation \eqref{NLW-p},
Grillakis \cite{grillakis1990regularity}  first  established
the global existence theory for classical solution  
when $d=3$.
 The results for  other dimensions are proved in  \cite{grillakis1992regularity, shatah1993regularity}.
scattering results for large energy data are proved in
\cite{BG-1999-AJM, BAHOURI1998DECAY, nakanishi1999unique} %
by establishing variants of the
Morawetz estimates
(see Morawetz \cite{morawetz1968time})
\beq\label{morawetz-sp=1}
\iint_{\R^{1+d} } \frac{|u|^{\frac{2d}{d-2}}}{|x|} dx dt \leq C_d E(u_0,u_1),
\eeq
where  $C_d$ is a  
constant depending on $d$.  
%
%
%
%
While in focusing energy critical cases, the Morawetz estimates
\eqref{morawetz-sp=1} fails.  And the scattering results do not
hold in general, since \eqref{NLW-p} has a ground state
$W(x)=\left(1+\frac{|x|^2}{d(d-2)} \right)^{-\frac{d-2}{2}}. $
In the cases of   $3\leq d \leq 5$, Kenig-Merle \cite{kenig2008global}
proved the scattering result for solution
with initial data  such that $E(u_0, u_1) < E(W,0)$ and $\norm{u_0}_{ \dot{H}^1 (\R^d)} < \norm{W}_{ \dot{H}^1 (\R^d)}$.
%
In their proofs, the main ingredient is the concentration compactness/rigidity
theorem method introduced by \cite{kenig2006global}.
This method is 
powerful and  plays an important role in study of
many other nonlinear dispersive equations.
We refer to \cite{killip2009nonlinear, Koch2014Dispersive, Kenig2015Lectures}
and the references therein.

%

For the defocusing subcritical equation \eqref{NLW-p},
the global existence has been proved for solution  with initial
data in the energy space $\HL(\R^d)$ by  Ginibre-Velo  \cite{Ginibre1985the, ginibre1989the}. 
However, there are no scattering results even for solutions with  initial datum in 
$(\dot H^1 \cap  \dot H^{s_p}) \times (L^2 \cap \dot H^{s_p-1})(\R^d)$.

Recently, Dodson  \cite{dodson-2016} proved  scattering results
for the defocusing cubic wave equation with
the initial datum belonging to the space
 $\dot B^{2}_{1,1} \times \dot B^{1}_{1,1}(\R^3)$,
which
is a subspace of $\HH(\R^3)$.
We remark that this is the first work
that gives scattering result for large data in the critical Sobolev space
without any a priori bound on the critical norm of the solution.  Dodson's strategy  
 consists of three steps: 
  \begin{enumerate}
    \item[(1)]  By establishing some new Strichartz-type estimates, 
    one can show that the solution  is in the energy space $\HL(\R^3)$ up to some 
    free evolutions.
       Then this decomposition enables one to prove the global well-posedness of the solution.

    \item[(2)]To obtain the scattering result, a conformal transformation is applied to
              show that the energy part of the solution has  finite energy in hyperbolic coordinates.
              Then from the conformal invariance of the equation and  a Morawetz-type inequality,  one can deduce that $\|u\|_{ L^4_{t,x}(\R\times \R^5)}\leq C(\norm{\uui}_{\BB(\R^5)},\delta_1),$
              where the parameter $\delta_1$ relies on the scaling and spatial profiles of the initial data.
    \item[(3)]   Finally, one can remove the dependence of $\delta_1$ by employing the profile
              decomposition, 
              which completes the proof. 
  \end{enumerate}


Let $S(t)(f,g)$ be the solution of Cauchy problem to the free wave equation
\beq
\label{FNLW}
\left\{
 \begin{aligned}
                 &\pa_{tt}v-\Delta v=0,\quad\,\qquad\qquad\;\, (t,x)\in \R\times \R^5, \\
               &(v, \pa_t v)\big|_{t=0} = (f(x), g(x)), \;\;\; x\in \R^5.
  \end{aligned}\right.
\eeq
For the sake of statement,  we introduce the  following notation as
$$\dot S(t)(f,g)\triangleq\pa_tS(t)(f,g), \;\;\text{and}\;\; \vec S(t)(f,g)\triangleq\big(S(t)(f,g),\dot S(t)(f,g)\big).$$
We consider the Cauchy problem of nonlinear wave equation
\beq
\label{NLW}
\left\{
 \begin{aligned}
                   & \pa_{tt}u-\Delta u+|u|u =\,0,\,\,\,\,\,\,\, (t,x)\in \R\times\R^d, \\
              & (u(0), \pa_t u(0)) =\, \uui,\;\;\;\;\;x\in \R^d,
  \end{aligned}\right.
\eeq
Our main result can be stated as:

\begin{theorem}\label{main-theorem}
For any radial initial  data $\uui\in\BB(\R^5)$, the solution $u$ to \eqref{NLW} is
globally well-posed and scattering, i.e.,
there exists $(u^\pm_{0},u_1^\pm)\in\HH(\R^5)$ such that
\beq\label{sct-def} 
\lim\limits_{t\ra \pm\infty}\norm{\left(u(t),\pa_t u(t)\right)-\vec S(t)(u^\pm_0,u^\pm_1)}_{\dot H^\frac12  \times \dot H^{-\frac12} (\R^5)}     \,\ra\, 0.
\eeq
Furthermore, there is a
function $A:[0,\infty)\ra [0,\infty)$, such that
\beq\label{u-L3}
\norm{u}_{L^3_{t,x}(\R\times\R^5)} \leq A(\norm{\uui}_{\BB(\R^5)}).
\eeq
\end{theorem}

\begin{remark}

 1.   This theorem extends the results of  \cite{dodson-2016} to the $5$-dimension case.
  Though the proof will utilize the strategy given in \cite{dodson-2016}, but it is highly nontrivial.
  \vskip 0.2cm
2. Unlike the $3$D case, the dispersive estimate(see \eqref{Dsp-est})
gives a decay in time of order $-2$, which may cause a logarithmic failure
when one estimates
$$\norm{u}_{L^{\frac43}_t L^4_x(J\times\R^5)}\,+\,\sup\limits_{t\in J}
\left( t^{\frac34}\norm{u}_{L^4_x(\R^5)} \right),$$
where $0\in J$ is a local time interval. We circumvent this difficulty by using the inhomogenous Strichartz
estimates in \cite{taggart2008inhomogeneous} and  prove
the global well-posedness of $u$.
  \vskip 0.2cm

3.    For the scattering result, by reductions, we need to bound the $L^3_{t,x}$ of $w$
on the light cone $\{|x| \leq t+\frac12\}.$
We will define the hyperbolic energy by rewriting the equation \eqref{NLW} as the form
\begin{equation*}
\pa_{tt}(r^2w) -\pa_{rr}(r^2w)=-2w-r^2 |w|w.
\end{equation*}
Observing that the additional term $2w$ and
the nonlinear term  $r^2 |w|w$ enjoy same sign,
we can bound the $L^3_{t,x}$ norm of $w$ by applying a Morawetz-type inequality,
if we assume the hyperbolic energy of $w$ is bounded.
  \vskip 0.2cm

4. To certifying  the above assumption,  we will make full use of the
formula \eqref{SL-formula-radial} for radial solution and the sharp Hardy inequality.
In contrast to the $3D$ case, 
some terms in formula \eqref{SL-formula-radial} seem more difficult to dealt with.
However, the integration domains of these terms  
are symmetric about the radius $r$, 
which also consists with the Huygens
principle.
This fact allows us  
apply the Hardy-Littlewood maximal functions to verifying the assumption. 
\end{remark}

\vskip0.4cm
Now, we give the outline of the proof.
By the Strichartz estimates and a standard
fix point argument, for initial data $\uui$,
there exists a maximal time interval $I\subset \R$ such that there exists a unique solution $u$ (see Definition \ref{solution} in Section 2)
to the equation \eqref{NLW}
on $I\times \R^5$.
We consider the global well-posedness by
developing some Strichartz-type estimates \eqref{2r-5/r,r}.
Utilizing the standard blowup criterion, 
we can show the global well-posedness of $u$.


Next, we  claim  the following proposition. 

\begin{proposition}\label{u-3-bounded}
For every radial initial data $\uui\in\BB(\R^5)$,  let $u$ be the corresponding solution to \eqref{NLW}.
Then there exists a parameter $\delta_1$ depending the initial data $\uui$ and  a function $A:[0,\infty)^2\ra [0,\infty)$
such that
 \beq\label{u-scatter}
 \norm{u}_{L^3_{t,x}(\R\times\R^5)} \leq A(\norm{\uui}_{\BB(\R^5)},\delta_1).
 \eeq
\end{proposition}
We  prove Theorem \ref{main-theorem} by employing  this and   
establishing  Proposition \ref{tnjlambda}, where the proof provides an alternate proof of
Lemma $6.2$ in \cite{dodson-2016}.

Finally, we need to prove Proposition \ref{u-3-bounded}.
From the partition  $u=v+w$,
it suffices to show the boundedness of $L^3_{t,x}$ norm of $w$.
We  prove the hyperbolic energy of $w$ is uniformly bounded.
Then, a Morawetz-type inequality 
yields that the $L^3_{t,x}$ norm of $w$
is bounded in the cone, which finishes the proof.

This paper is organized as follows: Section \ref{basics} gives
some tools from harmonic analysis and basic properties for the wave equation.
And, in Section \ref{local-theory} we give the decomposition of  $u$
and prove its global well-posedness. 
The existence of the function $A$ in \eqref{u-L3} is shown
in Section \ref{scattering} based on the Proposition \ref{u-3-bounded}.
%
Finally,  in Section \ref{hyper-scatter}, we complete the proof 
by showing Proposition
\ref{u-3-bounded}.

We end the introduction with some notations used throughout this paper.
We  use $\mathcal {S}( \R^d)$ to denote the space of Schwartz functions on $\R^d$.
For $1\leq p\leq \infty$, we define $L^p(\R^d)$ by the spaces of Lebesgue measurable functions with finite $L^p(\R^d)$-norm, which is defined by
$$\norm{f}_{L^p(\R^d)}=\Big(\int_{\R^d} |f(x)|^p dx\Big)^{1/p}, \;\;\text{for}\;\;1\leq p< \infty, $$
and
$ \norm{f}_{L^\infty(\R^d)} =\esssup_{x\in\R^d} |f(x)|.$
We let $\ell^p$ be  the  spaces of complex number sequences
$\{a_n\}_{n\in \mathbb{Z}}$ such that $\{a_n\}_{n\in \mathbb{Z} }\in l^p$
if and only if
 $$\norm{\{a_n\}}_{\ell^p_n(\mathbb{Z})}\triangleq \Big( \sum_n |a_n|^p \Big)^{1/p}<\infty,\;\;\text{for}\;\;1\leq p<\infty$$
  and
$ \norm{\{a_n\}}_{\ell^\infty_n(\mathbb{Z}) }:= \sup_n |a_n| <\infty $.
We use $X\lesssim Y$ to mean that there exists a constant $C>1$ such that
$X\leq C Y$, where the dependence of $C$ on the parameters will be clear from the context.
We use $X\thicksim Y$ to denote  $X\les Y$ and  $Y \les X$. $A\ll B$ denotes
there is a sufficiently large number $C$ such that $A\leq C^{-1} B $.

\section{Basic tools and some elementary properties for the wave equation}\label{basics}

In this section, we recall some tools from harmonic analysis and useful results for  the wave equation.
\subsection{Some tools from harmonic analysis}
Recall the Fourier transform of $f\in \mathcal {S}( \R^d)$ is defined by
$$\widehat{f}(\xi)=(2\pi)^{-d/2}\int_{\R^d } f(x)e^{-ix\xi} dx ,$$
which can be extended to Schwartz distributions naturally.
We will make frequent use of the Littlewood--Paley projection
operators. Specifically, we let $\varphi$ be a radial smooth function
supported on the ball $|\xi|\leq 2$ and equal to 1 on the ball
$|\xi|\leq 1$.  For $j\in \mathbb{Z} $, we define the
Littlewood--Paley projection operators by
\begin{align*}
\widehat{P_{\leq j}f}(\xi)&\,:=\,
\varphi(\xi/2^{j})\widehat{f}(\xi),
 \\ \widehat{P_{> j } f}(\xi)& \,:=\,
(1-\varphi(\xi/2^j))\widehat{f}(\xi),
\\ \widehat{P_{j}f}(\xi) &\,:=\,
(\varphi(\xi/2^j)-\varphi({\xi}/2^{j-1}))\widehat{f}(\xi).
\end{align*}


The Littlewood--Paley operators commute with derivative operators
and are bounded on the general Sobolev spaces. These operators also obey the
following standard Bernstein estimates:
\begin{lemma}[Bernstein estimates]
For $1\leq r\leq q\leq\infty$ and $s\geq 0$,
\begin{align*}
\big\||\nabla|^{\pm s} P_{j} f \big\|_{L_x^r(\R^d)} &\, \thicksim\,
2^{\pm js} \big\|
P_{j} f \big\|_{L_x^r(\R^d)},  \\
\big\||\nabla|^s P_{\leq j} f \big\|_{L_x^r(\R^d)} &\, \lesssim\, 2^{js}
\big\|
P_{\leq j} f \big\|_{L_x^r(\R^d)},  \\
 \big\| P_{> j} f \big\|_{L_x^r(\R^d)} &\, \lesssim\,  2^{-js} \big\|
|\nabla|^{s}P_{> j} f \big\|_{L_x^r(\R^d)}, \\
\big\| P_{\leq j} f \big\|_{L^q(\R^d)} &\, \lesssim\,
2^{(\frac{d}{r}-\frac{d}{q})j} \big\| P_{\leq j} f
\big\|_{L_x^r(\R^d)},
\end{align*}
where the fractional derivative operator $|\nabla|^\sigma$ is defined by $\widehat{|\nabla|^\sigma f}(\xi) =|\xi|^\sigma \widehat f(\xi)$, for $\sigma\in\R$.

 \end{lemma}
\begin{definition}[Homogeneous Besov spaces]
Let $s$ be a real number and $1\leq p,r \leq \infty.$
We denote the homogeneous Besov norm by
\beq
\norm{f}_{\dot B^s_{p,r} (\R^d)}\,:= \,\big\|  \{2^{js} \norm{P_j f}_{L^p(\R^d)}\} \big\| _{\ell^r_j(\mathbb{Z})},
\eeq
for $f\in\mathcal{S}(\R^d)$. Then the Besov space $\dot B^s_{p,r}(\R^d)$ is the completion
of the Schwartz function under this norm.
\end{definition}
%
%
%
%

We shall give the following radial Sobolev-type inequalities, which are analogous  to $3D$ cases established  in  \cite{dodson-2016}.
We denote radial derivative  by $\pa_r f(x)=(\frac{x}{|x|}\cdot \nabla) f(x)$ for any function $f$ defined on $\R^5$.

\begin{lemma}[Radial Sobolev-type inequalities in Besov spaces]\label{rad-sobolev}
 For any radial function $f\in\mathcal{S}(\R^5)$, we have


\beq\label{ra-sblv-infty-2}
\norm{|x|^2  f}_{L^\infty(\R^5)} \les   \norm{f}_{\dot{B}^{\frac12}_{2,1}(\R^5)}.
\eeq
Let $\uui\in\BB(\R^5)$ be a radial function, then we have

\begin{align}
\norm{\frac{1}{|x|^2} \pa_r u_0(x) }_{L^1_x(\R^5)} + \norm{\frac{1}{|x|}\pa_{rr} u_0(x) }_{L^1_x(\R^5)} + \norm{\frac{1}{|x|^3} u_0(x) }_{L^1_x(\R^5)}+\norm{|x|^3\pa_r u_0}_{L^\infty_x(\R^5)} \,&\les\, \norm{u_0}_{\dot B^{3}_{1,1}(\R^5)}, \label{rad-sobolev-u_0}\\
\norm{\frac{1}{|x|} \pa_r u_1(x) }_{L^1_x(\R^5)}  + \norm{\frac{1}{|x|^2} u_1(x) }_{L^1_x(\R^5)}+ \norm{|x|^3 u_1(x)}_{L^\infty_x(\R^5)}
\,&\les\,\norm{u_1}_{\dot B^{2}_{1,1}(\R^5)}.\label{rad-sobolev-u_1}
\end{align}
\end{lemma}

\begin{proof}
We first consider \eqref{ra-sblv-infty-2}. Since $f$ is radial, using polar coordinates, we have
\begin{equation}\begin{split}
 P_jf(|x|) = ~~  & ~~  P_j f(x) ~=~\int_{\R^5}  \widehat{P_j f}(\xi) e^{ix\xi} d\xi\\
                   =~~ &~~  \int_0^\infty  \int_{\mathbb{S}^{4} }  \widehat{P_j f}(r ) r^{4} e^{irx\omega} d\sigma(\omega) dr .
\end{split}\end{equation}
Recall the decay estimates of Fourier transform of the surface measure on the sphere
  $$\widehat{ d\sigma_{\mathbb{S}^{4}}}(\xi) \leq C( 1+|\xi|)^{-2},$$
which and  H\"older's inequality yield
\beq\label{pj-L-infty}
   |P_j f(|x|)| \,\les\,   \int_0^\infty |\widehat{P_j f}(r)| r^{2} |x|^{-2 }  dr\\
                \,\les  \,|x|^{-2 } 2^{\frac12j}  \norm{P_j f}_{L^2} .
  \eeq 
Then the  inequality \eqref{ra-sblv-infty-2} follows  from \eqref{pj-L-infty} and the definition of the Besov space.

Next, we consider \eqref{rad-sobolev-u_0} and \eqref{rad-sobolev-u_1}. By the density of Schwartz functions in $\BB(\R^5)$, we may assume that $u_0,u_1\in \mathcal{S}  (\R^5). $
We claim that it suffices to show

\begin{align}
 \norm{\frac{1}{|x|^2} \pa_r u_0(x) }_{L^1_x(\R^5)} + \norm{\frac{1}{|x|}\left| \Delta u_0(x) \right| }_{L^1_x(\R^5)} + \norm{\frac{1}{|x|^3} u_0(x) }_{L^1_x(\R^5)}\, \les& \,  \norm{u_0}_{\dot B^{3}_{1,1}(\R^5)},\label{u_0--1}\\
  \norm{\frac{1}{|x|} \pa_r u_1(x) }_{L^1_x(\R^5)}  + \norm{\frac{1}{|x|^2} u_1(x) }_{L^1_x(\R^5)} \, \les& \,
   \norm{u_1}_{\dot B^{2}_{1,1}(\R^5)}\label{u_1--1}.
\end{align}
To see this,  by using the fact $\Delta f=\pa_{rr}f+\frac4r \pa_r f$ for radial function $f(x)$ on $\R^5$,  we have
\beq
 \norm{\frac{1}{|x|}\pa_{rr} u_0(x) }_{L^1_x(\R^5)} \,\les\,  \norm{\frac{1}{|x|^2} \pa_r u_0(x) }_{L^1_x(\R^5)} + \norm{\frac{1}{|x|}\left| \Delta u_0(x) \right| }_{L^1_x(\R^5)} \,\les\, \norm{u_0}_{\dot B^{3}_{1,1}(\R^5)}
\eeq
From the fundamental theorem of calculus and polar coordinates, for $y\in\R^5\backslash\{0\} $, we have
\beq
|y|^3|\pa_r u_0(y)|
\,\les\, \int_{|y|}^\infty \int_{\mathbb{S}^4} r^3 |\pa_{rr} u_0(r)| d\sigma(\omega)dr
\,\les\,  \norm{\frac{1}{|x|}\pa_{rr} u_0(x) }_{L^1_x(\R^5)}\,\les \,
   \norm{u_0}_{\dot B^{3}_{1,1}(\R^5)}
\eeq
and
\beq
|y|^3 | u_1(y)| \,\les\, \int_{|y|}^\infty \int_{\mathbb{S}^4} r^3 |\pa_{rr} u_1(r)| d\sigma(\omega)dr
\,\les\,     \norm{\frac{1}{|x|} \pa_r u_1(x) }_{L^1_x(\R^5)}\,\les\, \norm{u_1}_{\dot B^{2}_{1,1}(\R^5)}.
\eeq

Hence, we are reduced to proving \eqref{u_0--1} and \eqref{u_1--1}.
We just give the estimate for the first term on the
 left hand side of \eqref{u_0--1}, since others can be handled  similarly.
For $j\in\Z$, utilizing Bernstein's estimates and polar coordinates, we obtain

\beq\begin{split}
         \norm{\frac{1}{|x|^2}\pa_r P_j u_0(x)}_{L^1_x(\R^5)}
     \les\,  &\, \int_0^\infty \int_{\mathbb{S}^4}     r^2 \pa_r P_j u_0(r) d\sigma(\omega) dr              \\
     \les\,  & \,  \int^{2^{-j}}_0  \frac{1}{r^2} |\pa_r (P_j u_0)| r^4 dr
               + \int_{2^{-j}}^{\infty}  \frac{1}{r^2} |\pa_r (P_j u_0)| r^4 dr \\
    \les \, &\, 2^{-3j} \norm{\pa_r(P_j u_0)}_{L^\infty_r(\R_+)} + 2^{2j}\norm{\pa_r(P_j u_0)}_{L^1_x(\R^5)} \\
 \les  \,&  \,  2^{3j} \norm{P_j u_0}_{L^1_x(\R^5)}.
    \end{split}
\eeq
Thus, we have $ \norm{\frac{1}{|x|^2}\pa_r  u_0(x)}_{L^1_x(\R^5)} \, \les\, \norm{u_0}_{\dot{B}^3_{1,1}(\R^5)}$.
\end{proof}
As a direct consequence of Lemma \ref{rad-sobolev}, we have
\beq\label{u_01-L^2}
\norm{|x|^\frac12 \pa_r u_0(x)}_{L^2_x(\R^5)}+\norm{|x|^{-\frac12} u_0(x)}_{L^2_x{(\R^5)}} +\norm{|x|^\frac12 u_1(x)}_{L^2_x(\R^5)} \les \norm{\uui}_{\BB(\R^5)}.
\eeq

\begin{lemma}\label{besov-a}
Suppose $\chi(x)\in C_c^\infty(\R^5).$ Let $R=2^k$ be a dyadic number for $k\in \mathbb{Z}$ and
denote $\chi_R(x)=\chi(\frac{x}R)$. Then  we have
\beq\label{multi-besov-1}
\norm{\chi_R(x) f}_{\dot{B}^{\frac12}_{2,1} (\R^5)}\, \les\, \norm{f}_{\dot{B}^{\frac12}_{2,1}(\R^5)},
\eeq
\beq\label{multi-besov-2}
\norm{\chi_R(x) g}_{\dot{B}^{-\frac12}_{2,1}(\R^5) }\, \les\, \norm{g}_{ \dot{B}^{-\frac12}_{2,1} (\R^5)},
\eeq
where the bound is independent of $R.$
Furthermore, if $\chi(x)=1$ on $|x|\leq 1$, then for
$(f,g)\in\dot{B}^{\frac12}_{2,1}\times \dot{B}^{-\frac12}_{2,1}(\R^5)$, we have

\beq\label{besov-lim}
\lim_{R\ra \infty} \norm{(1-\chi_R(x))f}_{\dot{B}^{\frac12}_{2,1}(\R^5)}
  +\norm{(1-\chi_R(x))g}_{\dot{B}^{-\frac12}_{2,1} (\R^5)} \, =\,0.
\eeq
\end{lemma}
\begin{proof}
By scaling, to prove the inequalities \eqref{multi-besov-1} and \eqref{multi-besov-2}, it suffices to prove the cases for $R=1,$
which follows from  a similar proof of Lemma 2.2 in  \cite{dodson-2016}.
 On the other hand, \eqref{besov-lim} follows from \eqref{multi-besov-1}, \eqref{multi-besov-2}, and the fact that  $C_c^\infty\times C_c^\infty(\R^5) $ is dense in $ \dot{B}^{\frac12}_{2,1}\times \dot{B}^{-\frac12}_{2,1}(\R^5).$
\end{proof}

Finally, we  need the following chain rule estimates for later use.

\begin{lemma}[$C^1$-fractional chain rule,\cite{christ1991dispersion}]\label{C1-chain}
 Suppose $G\in C^1(\mathbb{C}), s\in(0,1],  $ and $1<q,q_1,q_2<\infty$ satisfying
$\frac1q=\frac1{q_1}+\frac1{q_2}$. Then
\beq
\norm{|\nabla|^s G(u)}_{L^q(\R^d)}\, \les\,
 \norm{G'(u)}_{L^{q_1}(\R^d)} \norm{|\nabla|^s u}_{L^{q_2}(\R^d)}.
\eeq
\end{lemma}

\subsection{ Fundamental  properties of the wave equations}
Throughout the paper, by abuse of notations,  we often write
$u(t)=u(t,x)$ for simplicity and
$u(t,r)=u(t,x)$ when $u(t,\cdot)$ is radially symmetric.

Recall the explicit formula for solution to the linear wave equation in 5 dimension,
\beq\label{SL-formula}
\begin{split}
S(t)(f,g)(x)\,=\,&\cos(t|\nabla|)f(x)+\frac{\sin (t|\nabla|)}{|\nabla|}g(x) \\
        =\,& \frac{1}{3\omega_5}\pa_t\Big[\frac{1}{t}\pa_t\Big]\Big( t^3\int_{|y|=1} f(x+ty)d\sigma(y)\Big) + \frac{1}{3\omega_5}\frac{1}{t}\pa_t\Big(t^3 \int_{|y|=1} g(x+ty)d\sigma(y)\Big),
\end{split}
\eeq
where $\omega_5$ is the surface area of the unit sphere in $\R^5$.
When $(f,g)$ is radially symmetric, 
for $t> 0$, \eqref{SL-formula} can be rewritten as
\beq\label{SL-formula-radial}
\begin{split}
S(t)(f,g)(r)\,=\,&\frac1{2r^2}[(r-t)^2f(r-t)+(r+t)^2f(r+t)]
         -\frac{t}{2 r^{3}}\int^{r+t}_{|r-t|} s f(s) ds\\
        & +\frac1{4r^3} \int^{r+t}_{|r-t|} s  (s^2+r^2-t^2) g(s) ds.\\
\end{split}
\eeq
See also
\cite{rammaha1987finite-time, lindblad1996long, Colzani2002Radial}
for the radial solutions to general dimensions linear wave equation.
From the explicit formula \eqref{SL-formula}, we can obtain the following dispersive estimate.

\begin{proposition}[Dispersive estimate]
\beq\label{Dsp-est}
  \norm{S(t)(u_0,u_1)}_{L^\infty(\R^5)}\, \les\, \frac{1}{t^2}\big[\norm{ \nabla^3 u_0}_{L^1(\R^5)} + \norm{\nabla^2 u_1}_{L^1(\R^5)} \big].
\eeq
\end{proposition}
\begin{proof}

We give the proof  for completeness.
Similar proof for $3D$ case can be found in \cite{killip2011the}. 
By \eqref{SL-formula}, the free solution $S(t)(u_0,u_1)$ can be rewritten as 
\beq\label{explit}
\begin{split}
  & \frac{1}{\omega_5}\int_{|y|=1} u_0(x+ty) d\sigma(y)
+ \frac{5t}{3\omega_5}\int_{|y|=1} y(\nabla  u_0)(x+ty) d\sigma(y)
+\frac{t^2}{3\omega_5}\int_{|y|=1} y(\nabla^2  u_0)(x+ty) y d\sigma(y)\\
 & + \frac{t}{\omega_5} \int_{|y|=1}  u_1(x+ty) d\sigma(y)
 +  \frac{t^2}{3\omega_5} \int_{|y|=1} y (\nabla  u_1)(x+ty) d\sigma(y),
\end{split}
 \eeq
 which and the fundamental theorem of calculus yield \eqref{Dsp-est}.  For instance,  using polar coordinates, we can estimate the first term of \eqref{explit} as
\beq
\begin{split}
  \left|\frac{1}{\omega_5}\int_{|y|=1} u_0(x+ty) d\sigma(y)\right| \, &=\,  \left|\frac{1}{\omega_5} \int_{t}^\infty \int_{s}^\infty \int_{\tau}^\infty
   \int_{|y|=1}   \frac{d^3}{d\rho^3 }\big[u_0(x+\rho y)\big] d\sigma(y)  d\rho d\tau ds\right|\\
  \,& \les \,  \int_{t}^\infty \int_{s}^\infty    \int_{\tau}^\infty \int_{|y|=1} |\nabla^3 u_0|(x+ \rho y)  d\sigma(y) d\rho  d\tau ds\\
  & \les\, \int_{t}^\infty \int_{s}^\infty \frac{1}{\tau^4}  d\tau ds  \norm{\nabla^3 u_0}_{L^1(\R^5)}\\
  &\les\, \frac{1}{t^2}  \norm{\nabla^3 u_0}_{L^1(\R^5)}.
\end{split}
\eeq
And other terms can be dealt with similarly.

\end{proof}

We recall the Strichartz estimates of wave equation in $\R^5$.
Let $I\subset\R$ be an interval. We denote the spacetime norm $L^q_tW^{s,r}_x(I\times \R^5)$ of a function $u(t,x)$ on $I\times \R^5$ by
 $$ \norm{u}_{L^q_tW^{s,r}_x(I\times \R^5)}\,:=\,\big\|  \norm{u(t,x)}_{W^{s,r}_x(\R^5)}  \big\|_{L^q_t(I)}, $$
 for $s\in\R$, $1\leq q,\,  r\leq \infty.$  We denote that a pair $(q,r)$ of exponents is  admissible, if

%

 \beq
 2\leq q \leq \infty,\, 2\leq r<\infty,\,   \text{ and } \frac1q+\frac{2}{r}\leq 1 .
 \eeq
Moreover, we say $(q,r)$ is wave acceptable, provided
\beq
1\leq q<\infty,\,  2\leq r\leq \infty, \, \frac1q<4 (\frac12-\frac1r),
\eeq or $(q,r)=(\infty,2)$.

\begin{proposition}[Strichartz estimates \cite{lindblad1995on,ginibre1995generalized,keel1998endpoint}]\label{stri-prop}
Let  $\uui\in \HH(\R^5)$ and $(q,r)$, $(\tq,\tr)$ be two admissible pairs.
If $u$ is a weak solution to the wave equation $\pa_{tt} u-\Delta u =F(t,x) $ with initial data $\uui$, then we have
\beq\label{stri}
              \norm{|\nabla|^\rho u}_{L^q_tL^r_x(I\times\R^5)}+\sup_{t\in I}\norm{\uu}_{\HH(\R^5)}
  \lesssim\,\,  \norm{\uui}_{\HH(\R^5)} +\norm{|\nabla|^{-\mu}F}_{L^{\tq'}_tL^{\tr'} (I\times\R^5)},
\eeq
%
provided that
\beq
\rho\,=\, \frac1q+\frac{5}{r}-2  \text{\,\, and\,\, } \mu\,=\,\frac{1}{\tq}+\frac{5}{\tr}-{2}.
\eeq
\end{proposition}

\begin{proposition}[Inhomogeneous Strichartz estimates \cite{taggart2008inhomogeneous}]
Suppose the exponent pairs $(q_1,r_1)$ and $(\tq_1,\tr_1)$ are wave
acceptable, and  satisfy the scaling condition

$$\frac{1}{q} +\frac1{\tq}\,=\,2 -2(\frac1{r_1}+\frac1{\tr_1}) $$
and the conditions
$$\frac{1}{q} +\frac1{\tq}<1,\ \   \frac12\leq  \frac{\tilde r_1}{r_1} \leq 2. $$
Let $r\geq r_1,\ \tr\geq\tr_1,\ \rho\in\R$, and such that
$$\rho+5(\frac12-\frac1r)-\frac1q=1-(\tilde{\rho}+  5(\frac12-\frac1{\tr})-\frac1{\tq}).  $$
If $F(t,x)\in L_t^{\tq'}(\R; \dot{B}^{-\tilde{\rho}}_{\tr',2}(\R^5)) $
and  $u$ is a weak solution to the inhomogeneous wave equation
\beq
-\pa_t^2 u+\Delta u=F(t,x),\ \ u(0)=u_t(0)=0,
\eeq
then

\beq\label{inhom-stri}
\norm{u}_{L_t^q(\R; \dot{B}^\rho_{r,2} (\R^5))}  \lesssim \|F(t,x)\|_{ L_t^{\tq'}(\R; \dot{B}^{-\tilde{\rho}}_{\tr',2}(\R^5))}.
\eeq

\end{proposition}
Next, we recall the well-posedness theory and the  perturbation theory of the Cauchy problem  \eqref{NLW}.
\begin{definition}[Solution]\label{solution}  Let $I $ be a time interval such that  $0\in I$.
We call function $u:I\times \R^5\ra \R$ is a (strong) solution to
the Cauchy problem \eqref{NLW} in $I$ if it satisfies
$\uu(0)=\uui$, $$ \uu\in C(I; \HH(\R^5))\,\cap\,L^{3}_{t,x}(I\times \R^5),$$ and
 the integral equation
\beq
u(t) =S(t)\uui -\, \int_0^t S(t-\tau)(0,|u|u(\tau)) d\tau
\eeq
for all $t\in I$.

\end{definition}

\begin{theorem}[Local well-posedness \cite{lindblad1995on, casey-2016}]  \label{lwp}
Let $\uui\in\HH(\R^5)$ with $$\norm{\uui}_{\HH(\R^5)}\leq A . $$
There exists $\delta=\delta(A)>0 $ such that, if
\beq
\norm{S(t)\uui}_{L^{3}_{t,x}([0,T]\times\R^5)} \leq \delta, \text{\,\, for some\,\, } T>0,
\eeq
then there exists a unique solution $u$ to \eqref{NLW} in $[0,T]\times\R^5$, such that
\beq
\sup_{0\leq t\leq T}\norm{(u,u_t)}_{ \HH(\R^5)}+\norm{u}_{L^{3}_{t,x}([0,T]\times\R^5)} \leq C(A).
\eeq
In addition, if $A>0$ is small enough, we can take $T=\infty.$

\end{theorem}
We define $T_+\uui:=\sup\, I$, where  $I$ is the
maximal interval of existence of the solution $u$.

\begin{lemma}[Standard blowup criterion]\label{blow-criterion}
Suppose $u$ is the solution to the Cauchy problem \eqref{NLW}
with initial data $\uui\in\HH(\R^5)$ and $T_+\uui<\infty$. Then we have

\beq
\norm{u}_{L^{3}_{t,x}([0,T_+\uui)\times\R^5)}=\infty.
\eeq
\end{lemma}
 The proof is standard and similar to the energy critical case in \cite{Kenig2015Lectures}.
\vskip.3cm

We end this section by recalling the stability lemma for the Cauchy problem \eqref{NLW}, which plays an important role in the Section \ref{profile}.

%
%
%
%
%
%

\begin{theorem}[Perturbation theory \cite{casey-2016}]\label{perturbation}
Let $I\subset\R$ be a time interval with $0\in I$. Let $\uui\in\HH(\R^5)$ and some constants
$M,A,A'>0$ be given. Let $\tu$ be defined on $I\times \R^5$
and satisfy
\beq
\sup_{t\in I} \norm{(\tilde u, \pa_t \tilde u)}_{\HH(\R^5)}\leq A
\eeq
and
\beq
\norm{\tu}_{L^3_{t,x}(I\times\R^5)} \leq M.
\eeq
Assume that
$ \pa_{tt}\tilde u -\Delta \tilde u =- | \tilde u| \tilde u+e$
on $I\times \R^5$,
\beq
\norm{(u_0-\tu(0),u_1-\pa_t\tu(0) )}_{\HH(\R^5)} \leq A'
\eeq
and that
\beq
\norm{e}_{L^{3/2}_{t,x}(I\times\R^5)} + \norm{S(t)[(\tilde u(0),\pa_t \tilde u(0) )-\uui]}_{L^{3}_{t,x}(I\times\R^5)}<\vps.
\eeq
 Then, there exist $\beta>0$ and
$\vps_0=\vps_0(M,A,A')>0$ such that,  for $0<\vps<\vps_0$,  there exists a solution $u$ to \eqref{NLW} in $I$ such that  $(u(0),\pa_t u(0))=\uui$, with
 \beq
\norm{u}_{L^3_{t,x}(I\times\R^5)}\leq C(M,A,A'),
\eeq
and
\beq
\sup_{t\in I}\norm{(\pa_t\tilde u(t), \tilde u(t))-(u,\pa_t u(t))}_{\HH(\R^5)} \leq C(M,A,A')(A'+\vps^\beta).
\eeq
\end{theorem}

\section{Decomposition of the solution and global well-posedness}\label{local-theory}

In this section, we will prove that for any given initial data $\uui\in\BB(\R^5)$, then  the corresponding solution $u$ to the equation \eqref{NLW} is globally well-posed. To prove this, we first show the solution $u$
belongs to some suitable Strichartz-type spaces on a local time interval. Then, we split it into two parts: $u=v+w$.
Based on the inhomogeneous Strichartz estimates \eqref{inhom-stri}, we will derive a decay property  for $v$ and prove that  $w$ is in the energy space $\HL(\R^5)$. 
We remark that  the constants in `` $\les$ '' in this section depend upon $\norm{\uui}_{\BB(\R^5)} $.

For the sake of simplicity, we denote $F(u)=|u|u$. 
Recall that $u_\lambda $ is also a solution to the Cauchy problem \eqref{NLW} with initial data
$(\lambda^2u_0(\lambda x),\lambda^3 u_1(\lambda x))$, where
\beq\label{scal}
u_\lambda(t,x)=\lambda^2u(\lambda t,\lambda x).
\eeq
Given $\uui\in\BB(\R^5)$, for any $\eta>0$, there exists $j_0=j_0(u_0,u_1,\eta)<\infty$ such that
\beq\label{}
\sum_{j\geq j_0} 2^{2j}\norm{P_j u_0}_{L^1(\R^5)}<\eta.
\eeq
Replace  $u$ by $u_\lambda$ for 
$\lambda =2^{-j_0}$, then we have

\beq\label{low-frq}
\sum_{j\geq 0} 2^{3j}\norm{P_j u_0}_{L^1(\R^5)}+  \sum_{j\geq 0} 2^{2j}\norm{P_j u_1}_{L^1(\R^5)}<\eta.
\eeq

\begin{lemma}
\label{local-1}
Let $\ep_0>0$ be a small constant and $\eta\ll \ep_0$. If the initial data $\uui\in\BB(\R^5)$ satisfies \eqref{low-frq} and $u$ is the solution to \eqref{NLW} with initial data $\uui$ given by Theorem \ref{lwp},  then there exists $$\delta=\delta(\ep_0,\norm{\uui}_{\BB})>0$$ such that
\beq
\norm{ u}_{L_{t,x}^3 ([-\delta,\delta]\times\R^5)}  \leq \sum_{j\in \Z} \norm{P_j u}_{L^3_{t,x} ([-\delta,\delta]\times\R^5)}\les\ep_0,
\eeq
\beq\label{u-BB}
 \norm{u}_{L_t^\infty ([-\delta,\delta];\bt (\R^5) )} \les \sum_{j\in\Z} \norm{P_j u}_{L^\infty_t\dot H^{1/2}([-\delta,\delta]\times\R^5)} \les 1. 
\eeq
\end{lemma}
\begin{proof} 
By Strichartz's estimates in Proposition \ref{stri-prop}  and \eqref{low-frq}, we obtain

\beq
\norm{S(t)P_{\geq 0}\uui }_{L^{3}_{t,x}(\R \times \R^5)} \leq \frac12 \ep_0.
\eeq
On the other hand, for every $t\in\R$,  by Bernstein, we have
\beq
\norm{S(t)P_{\leq 0}\uui}_{L^3_{x}(\R^5)} \,\les\, 1.
\eeq
Hence, taking $\delta$ small enough, we have,
\beq
\norm{S(t) \uui }_{L^{3}_{t,x}([-\delta,\delta] \times \R^5)} \leq \ep_0.
\eeq
Then, by the Strichartz estimates, we have
\beq
\norm{u}_{L^3_{t,x}([-\delta,\delta] \times \R^5)} \les \norm{S(t)(u_0,u_1)}_{L^{3}_{t,x}([-\delta,\delta] \times \R^5)}
 + \norm{u}_{L^3_{t,x}([-\delta,\delta] \times \R^5)}^2,
\eeq
from which,  by a standard  continuity  argument, we deduce that
\beq
\norm{u}_{L^{3}_{t,x}([-\delta,\delta] \times \R^5)}\les \ep_0.
\eeq

Let
\beq
a_k=\norm{P_k u}_{L^{3}_{t,x}([-\delta,\delta] \times \R^5)} + 2^{\frac{1}{2}k}\norm{P_k u}_{L^\infty_tL^2_x([-\delta,\delta] \times \R^5)}+2^{\frac{1}{4}k}\norm{P_k u}_{L^6_tL^\frac{12}{5}_x([-\delta,\delta] \times \R^5)},
\eeq
\beq
b_k=2^{\frac{k}{2}}\norm{P_ku_0}_{L^2_x}+2^{-\frac{k}2} \norm{P_k u_1}_{L^2_x}.
\eeq
By the Young inequality,  it suffices to show there is a recurrence relation
\beq\label{rerc-1}
a_k\les b_k+\ep_0 \sum_{j} 2^{-\frac{1}{4} |j-k|} a_j.
\eeq
To prove this, making use of  the Strichartz estimates, we have
\beq\label{rerc-3}
a_k\les b_k +2^\frac{k}{4} \norm{P_k F(u)}_{L^2_t L^\frac43_x ([0,\delta]\times\R^5)}.
\eeq

First, we consider the low frequency part of the second term in the right hand of \eqref{rerc-3}. By Lemma \ref{C1-chain} and H\"older, we have
\beq\begin{split}
       &\norm{P_k F(u_{\leq k})}_{L^2_t L^\frac43_x ([-\delta,\delta] \times \R^5)}\\
  \les &\, 2^{-\frac12 k}\norm{P_k |\nabla_x|^{\frac12}  F(u_{\leq k}) }_{L^2_t L^\frac43_x ([-\delta,\delta] \times \R^5)}\\
  \les &\,2^{-\frac12 k} \norm{u}_{L^3_{t,x}([-\delta,\delta] \times \R^5)} \norm{|\nabla|^{\frac12} (P_{\leq k}u)}_{L^6_tL^\frac{12}5_x([-\delta,\delta] \times \R^5)}\\
  \les &\, \norm{u}_{L^3_{t,x} ([-\delta,\delta] \times \R^5)}
  \sum_{j\leq k} 2^{-\frac{k-j}{2}} \norm{P_{j}u}_{L^6_tL^\frac{12}5_x([-\delta,\delta] \times \R^5)},
\end{split}
\eeq
from which it follows that
\beq\label{rec-1}
\begin{split}
  &2^{\frac{k}4}\norm{P_k F(u_{\leq k}) }_{L^2_t L^\frac43_x([-\delta,\delta] \times \R^5) }  \\
\les& \ep_0 \sum_{j\leq k} 2^{-\frac{k-j}4} 2^{\frac14 j }  \norm{P_j u}_{L^6_tL^\frac{12}{5}_x ([-\delta,\delta] \times \R^5)} \les \ep_0 \sum_{j\leq k} 2^{-\frac{k-j}4} a_j.
\end{split}
\eeq
On the other hand, by H\"older's inequality,
\beq\label{rec-2}\begin{split}
& 2^{\frac{k}4} \norm{P_{k} (F(u)- F(u_{\leq k}) )}_{L^2_tL^{\frac43}_x ([-\delta,\delta] \times \R^5)}\\
\les & \norm{u}_{L^3_{t,x} ([-\delta,\delta] \times \R^5)} 2^{\frac14 k} \norm{P_{\geq k+1} u}_{L^6_tL^\frac{12}{5}_x ([-\delta,\delta] \times \R^5)}\\
\les & \ep_0 \sum_{j\geq k+1} 2^{-\frac{j-k}{4}} 2^{\frac14j} \norm{ P_j u}_{L^6_tL^\frac{12}{5}_x ([-\delta,\delta] \times \R^5)} \les  \ep_0 \sum_{j\geq k+1} 2^{-\frac{j-k}{4}} a_j.
\end{split}
\eeq
Then the recurrence relation \eqref{rerc-1}  follows from \eqref{rec-1} and \eqref{rec-2}.
\end{proof}

Remark  that by the inequality \eqref{u-BB}, the inequalities \eqref{rec-1} and \eqref{rec-2} yield that
\beq\label{Pk Fu}
\sum_{k\in \mathbb{Z} } 2^\frac{k}{4} \norm{P_k \big(F(u)\big)}_{L^2_t L^\frac43_x ([0,\delta]\times\R^5)}\,\les\, \ep_0 .
\eeq


As an application of Lemma \ref{local-1} and the radial Sobolev
inequality \eqref{ra-sblv-infty-2}, we will see that the solution $u$ posses
some space decay property  in the region
$\{|x|\geq |t|+C\}$ for some large constant $C>0$.  
Let $\chi(x)$ be a smooth cutoff function
such that $\chi(x)=1$ for $|x|\leq \frac12 $ and $\chi(x)=0$ for $|x|\geq 1$.
By Lemma \ref{besov-a}, there exists a dyadic integer $R=R(u_0,u_1,\ep_0)$ such that
$
\norm{(1-\chi(\frac{x}{R}))\uui}_{\bb(\R^5)} \leq \ep_0.
$
Then by scaling, we have
\beq
\norm{\big(1-\chi(2x)\big)    \big( (2R)^2 u_0(2Rx),  (2R)^3 u_1(2Rx) \big) }_{\bb(\R^5)} \leq \ep_0.
\eeq
By abuse of notations, we will still use  $\uui$ to represent the initial data $\big( (2R)^2 u_0(2Rx),  (2R)^3 u_1(2Rx) \big)$.
Then we have,
\beq\label{ui-infy-small}
\norm{\big(1-\chi(2x)\big)\uui}_{\bb(\R^5)} \leq \ep_0.
\eeq
In addition, by Lemma \ref{local-1}, we have
\beq
\norm{u}_{L^3_{t,x}([-\frac{\delta}{2R},\frac{\delta}{2R}]\times \R^5)}\les \ep_0,
\eeq

\beq\label{u_H<B}
\norm{u}_{L_t^\infty ([-\frac{\delta}{2R},\frac{\delta}{2R}] ;\dot{B}_{2,1}^{\frac12}(\R^5) )}\,\les\, 1.
\eeq

\begin{lemma}\label{u-outcone}
Let $J\subset \R$ be an interval such that  $u$ is a solution to \eqref{NLW} on $J$. Then we have

  \beq\label{u-out-cone}
\norm{u}_{L^3_{t,x}(\{(t,x)\in J\times\R^5:\,|x|\geq |t|+\frac12 \})} + \sup_{t\in J } \norm{|x|^2 u(t,x)}_{L^\infty_x(\{x\in\R^5:\,|x|\geq |t|+\frac12\})} \les \ep_0.
\eeq

\end{lemma}

\begin{proof}
Let $U(t,x)$ be the solution to \eqref{NLW} with initial data
$\big(1-\chi(2x)  \big)(u_0(x)  ,u_1(x) ).$
Employing Theorem \ref{lwp} and arguing by similar arguments in Lemma \ref{local-1},
one can deduce  \eqref{rerc-3} 
when $u$ is replaced by $U$ and $[-\delta,\delta]$ is replaced by $\R$. Thus, we have
%
%
%

\beq\label{U}
\norm{U}_{L^3_{t,x}(\R\times\R^5)}+\norm{U }_{L^\infty_t \dot B^{\frac12}_{2,1}(\R\times\R^5) } \, \les\, \ep_0.
\eeq
Due to the finite propagation speed
property of the wave equation \eqref{NLW},  we have $u(t,x)=U(t,x)$ when $|x|\geq |t|+\frac12.$
Then \eqref{u-out-cone} follows from \eqref{U} and  the radial Sobolev inequality \eqref{ra-sblv-infty-2}.
%

\end{proof}

Next, we want to show the following local properties, which will play an important
role in the following Subsection \ref{energy part}.  Unlike the case of  the three dimension in \cite{dodson-2016}, we will make use of the
inhomogeneous Strichartz estimates \eqref{inhom-stri} to conquer the difficulties caused by the higher order decay of time.

%
\begin{lemma}\label{local-2}
If $\ep_0$ is sufficiently small and $\delta$  is  
given in Lemma \ref{local-1},
then,  for  $3<r<4$, we have
\beq\label{sup-2r-5/r,r}
\sup_{-\frac{2\delta}{2R}<t<\frac{2\delta}{2R}}   t^{ \frac{2r-5}{r} }  \norm{u}_{L^{r}_x(\R^5)} +\norm{u}_{L^\frac{5}{4}_tL^{\frac{25}{6}}_x ([0,\frac{\delta}{2R}]\times \R^5) }  \,\les\, 1.
\eeq
\end{lemma}

We remark that for  $3<r<\infty$, the space $L^{\frac{r}{2r-5}}_t L_x^r(\R\times\R^5)$ is $\dot{H}^{\frac12}$-critical but not admissible.  

\begin{proof}

First, we consider the estimates for the linear part. Utilizing  dispersive estimate \eqref{Dsp-est},
we have
\beq\label{dispersive-2}
\norm{S(t)P_j\uui}_{L^\infty_x(\R^5)} \les \frac{1}{t^2} \big[2^{3j}\norm{P_j u_0}_{L^1(\R^5)} +2^{2j}\norm{P_j u_1}_{L^1(\R^5)} \big].
\eeq
By Bernstein, we have
\beq
\norm{S(t)P_j\uui}_{L^2_x(\R^5)} \les \norm{P_j u_0}_{L^2_x(\R^5)}+2^{-j}\norm{P_j u_1}_{L^2_x(\R^5)} \les 2^{\frac52j}\norm{P_j u_0}_{L^1(\R^5)} +2^{\frac32j}\norm{P_j u_1}_{L^1(\R^5)}.
\eeq
Interpolating this inequality with the estimate \eqref{dispersive-2} yields that,

\beq\label{r-1}
\norm{S(t)P_j\uui}_{L^r_x (\R^5) } \,\les\,  t^{-2(1-\frac2r) }  2^{-\frac{j}{r}  }  \big[2^{3j}\norm{P_j u_0}_{L^1(\R^5)}+2^{2j}\norm{P_j u_1}_{L^1(\R^5)}\big].
\eeq
On the other hand, for $r\geq \frac52$, employing  Bernstein' s estimates,  we have
\beq\label{r-2}\begin{split}
  \norm{S(t)P_j\uui}_{L^r_x (\R^5)}  &\,\les\,  2^{5(\frac12-\frac1r)j} \norm{S(t)P_j\uui}_{L^{2}_x (\R^5)}\\
 & \,\les\, 2^{(2-\frac5r)j} \big[2^{3j}\norm{P_j u_0}_{L^1(\R^5)}+2^{2j}\norm{P_j u_1}_{L^1(\R^5)}\big].
\end{split}
\eeq
This estimate and  \eqref{r-1} yield that, for $r\geq \frac52$,
\beq\label{2r-5/r,r}
\sup_{t\in\R} t^{ \frac{2r-5}{r} }  \norm{S(t)\uui}_{L_x^r(\R^5)}  +\norm{S(t)\uui}_{L^{\frac{r}{2r-5}}_t L_x^r(\R\times\R^5)}    \les  \norm{\uui}_{\BB(\R^5)}.
\eeq

By the reversal property of the wave equation, it suffices to prove  \eqref{sup-2r-5/r,r} for $t\geq0$.
Using the inhomogeneous Strichartz estimates \eqref{inhom-stri}, Lemma \ref{C1-chain}, Lemma \ref{local-1}, and H\"older, we have

\beq
\begin{split}
&\norm{\int^t_0 S(t-\tau)(0,F(u(\tau))) d\tau}_{L^\frac{5}{4}_tL^{\frac{25}{6}}_x ([0,\frac{\delta}{2R}]\times \R^5)}\\
& \les \norm{|\nabla|^{\frac14} F(u)}_{L_t^\frac{30}{29} L_x^{ \frac{300}{197} } ([0,\frac{\delta}{2R}]\times \R^5 ) }\\
&\les\norm{|\nabla|^\frac12 u}_{L_t^\infty L^2_x ([0,\frac{\delta}{2R}]\times \R^5)}^{1/2} \norm{u}_{L^3_{t,x} ([0,\frac{\delta}{2R}]\times \R^5)}^{1/2} \norm{u}_{L^\frac{5}{4}_tL^{\frac{25}{6}}_x ([0,\frac{\delta}{2R}]\times \R^5) }.
\end{split}
\eeq
This estimates together with the estimate \eqref{2r-5/r,r} yields
\beq\label{5/4,25/6}
\norm{u}_{L^\frac{5}{4}_tL^{\frac{25}{6}}_x ([0,\frac{\delta}{2R}]\times \R^5)} \,\les\,1,
\eeq
provided $0<\ep_0\ll \min(1,\norm{\uui}_{\BB(\R^5)}^{-1})$. 

Let $c\in(0,1)$  to be chosen later. 
First, employing the dispersive estimate \eqref{Dsp-est}, Lemma \ref{C1-chain}, and interpolation,  for $r\in(3,4)$, we have
\beq\label{0,c}
\begin{split}
         & \sup_{t\in[0,\frac{\delta}{2R}]} t^{\frac{2r-5}{r}} \norm{\int_{(1-c)t}^t S(t-\tau)(0, F(u(\tau)) ) d\tau}_{L^r(\R^5)} \\
  \les\,  &\,  \sup_{t\in[0,\frac{\delta}{2R}]}  t^{\frac{2r-5}{r}}  \int_{(1-c)t}^t \frac{1}{(t-\tau)^{2-\frac4r}  }
  \norm{ |\nabla|^{2-\frac6r} F(u(\tau)) }_{L^{r'}(\R^5)}  d\tau\\
  \les\, &\, \sup_{t\in[0,\frac{\delta}{2R}]}  t^{\frac{2r-5}{r}}  \int_{(1-c)t}^t \frac{1}{(t-\tau)^{2-\frac4r}  } \norm{u(\tau)}_{L^{r}(\R^5)}^{\frac{r-1}{2r-5}}
  \norm{u(\tau)}^{4-\frac{12}r}_{\dot{H}^{\frac12}_x (\R^5)  }
  \norm{u(\tau)}_{ L^\frac52_x(\R^5)}^{(\frac{12}r-2-\frac{r-1}{2r-5}) }
    d\tau \\
\les \,   &\, \Big(\sup_{t\in[0,\frac{\delta}{2R}]}   t^{\frac{2r-5}{r}}  \norm{u(t)}_{L^{r}(\R^5)}\Big)^{\frac{r-1}{2r-5} }\ \      \norm{u}^{4-\frac{12}r}_{L^\infty_t\dot{H}^{\frac12}_x ([0,\frac{\delta}{2R}]\times \R^5) }  \norm{u}_{L^\infty_t L^\frac52_x([0,\frac{\delta}{2R}]\times \R^5) }^{(\frac{12}r-2-\frac{r-1}{2r-5})}  \cdot \int_{(1-c)t}^t \frac{1}{(t-\tau)^{2-\frac4r} } t^{1 -\frac4r} d\tau\\
\les\, &\, c^{4/r-1} \Big(\sup_{t\in[0,\frac{\delta}{2R}]}   t^{\frac{2r-5}{r}}  \norm{u(t)}_{L^{r}(\R^5)}\Big)^{\frac{r-1}{2r-5}} .
\end{split}
\eeq

For the remainder part, we utilize the dispersive estimate \eqref{Dsp-est}, Lemma \ref{C1-chain}, the H\"older inequality, and interpolation  to obtain
\beq
\begin{split}
          &\, t^{\frac{2r-5}{r}} \norm{\int_0^{(1-c)t} S(t-\tau)(0, F(u(\tau)) ) d\tau}_{L^r(\R^5)}\\
  \les \, &\, t^{\frac{2r-5}{r}} \int_0^{(1-c)t} \frac{1}{(t-\tau)^{2-\frac4r}  }
  \norm{ |\nabla|^{2-\frac6r} F(u(\tau)) }_{L^{r'}_x(\R^5)}  d\tau\\
  \les\, &\, c^{\frac4r-2} \norm{\|\nabla|^{2-\frac6r} F(u)}_{L^{r'}_{t,x}(\R\times\R^5)} \\
\les\,    &\, c^{\frac4r-2}   \norm{u}^{}_{L^{\frac{r}{2r-5}}_tL^{r}_x (\R\times\R^5) } \norm{u}^{4-\frac{12}r}_{L^\infty_t \dot H^{\frac12}_x (\R\times\R^5)} \norm{u}^{\frac{12}r-3}_{L^3_{t,x} (\R\times\R^5)} \\
\les\, &\, c^{\frac4r-2}  \ep_0 ^{\frac{159}{7r}-\frac{39}{7}}  \label{1-c,c}, 
\end{split}
\eeq
where in the last step we used the fact that
 $$ \norm{u}^{}_{L^{\frac{r}{2r-5}}_tL^{r}_x ([0,\frac{\delta}{2R}]\times \R^5)}\les \norm{u}_{L^3_{t,x} ([0,\frac{\delta}{2R}]\times \R^5)}^{\frac{75}{7}\left(\frac1r-\frac{6}{25}\right)} \norm{u}^{\frac{25}{7}\left(1-\frac3r\right)}_{L^\frac{5}{4}_tL^{\frac{25}{6}}_x ([0,\frac{\delta}{2R}]\times \R^5)}\,\les\,\ep_0^{\frac{75}{7}(\frac1r-\frac{6}{25})} .$$
Hence, by \eqref{2r-5/r,r} and \eqref{5/4,25/6}-\eqref{1-c,c}, taking $c>0 $ small enough and using a continuity method, 
there exists $ \ep_0=\ep_0(\norm{\uui}_{\BB})>0$ such that
\beq
\sup_{t\in[0,\frac{\delta}{2R}]}  t^{\frac{2r-5}{r}} \norm{u}_{L^r_x} \les 1, 
 \text{ for } 3<r<4.
\eeq

\end{proof}

We denote $\delta_1=\frac{\delta}{2R}$ for simplicity.
Let
$\psi\in C^\infty_0(\R^5)$  supported in  $|x|\leq \frac{\delta_1}{10}$  and  $\psi(x)=1$ when $|x|\leq \frac{\delta_1}{20}$.
Then we can assume that $|(\nabla \psi)(x)| \les \frac{1}{\delta_1}$.
For $t\geq \delta_1$, we split $u(t,x)=v(t,x)+w(t,x),$ where
\beq\label{v(t)}
v(t) =S(t)(\psi u_0, \psi u_1) -\int_0^{\delta_1/10 } S(t-\tau)(0,\psi F(u(\tau))) d\tau.
\eeq
We will prove that $v$ has a decay property and $w$ is with finite energy.

\subsection{The decay part of the solution $u$}
\begin{lemma}\label{v-property}
 For $t\geq \delta_1$, we have
 \beq\label{v-dispersive}
 \norm{v(t)}_{L^\infty_x(\R^5 )}\les \delta_1^{-\frac12} t^{-2}.
 \eeq
In addition, we have
\beq\label{v-stri}
\norm{v}_{L^\infty_t(\R; \dot B^{\frac12}_{2,1}(\R^5)) }+ \norm{v}_{L^3_{t,x}(\R\times\R^5)}\les  \delta_1^{-\frac12}.
 \eeq

\end{lemma}
\begin{proof}

We first estimate the linear part of $v$. By the Huygens principle, the radial Sobolev inequality \eqref{ra-sblv-infty-2} and Lemma \ref{besov-a}, we have, for $t\geq\delta_1$
\beq\label{v-free-dispersive-infty}
\norm{S(t)(\psi u_0, \psi u_1)}_{L^\infty_x(\R^5)} \les \frac{1}{t^{2}}  \norm{ \uui }_{\dot{B}^\frac12_{2,1}\times \dot{B}^{-\frac12}_{2,1}} \les \frac{1}{t^{2}}. 
\eeq

For the second part of $v$ and $t\geq \delta_1$, using the Radial Sobolev inequality \eqref{ra-sblv-infty-2}, the Huygens principle and the  Strichartz estimates, we obtain
\begin{align}
        &\,\norm{\int_0^{\delta_1/10 } S(t-\tau)(0,\psi F(u(\tau))) d\tau}_{L^{\infty}_x(\R^5)}  \nonumber\\
  \les\,  &  \, \frac{1}{t^{2}}  \norm{ \int^{\delta_1/10}_0 S(t-\tau)(0,\tc F(u(\tau))) d\tau }_{\dot B^{\frac12}_{2,1}(\R^5)} \nonumber\\
  \les &\,\frac{1}{t^2}\norm{\int^{\delta_1/10}_0 \frac{\sin (\tau |\nabla|)}{|\nabla|} (\psi F(u(\tau))) d\tau }_{\dot B^{\frac12}_{2,1} }
+\frac{1}{t^2}\norm{\int^{\delta_1/10}_0 \frac{\cos (\tau |\nabla|)}{|\nabla|} (\psi F(u(\tau))) d\tau }_{\dot B^{\frac12}_{2,1}}\label{v-inf-duhamel}\\
\les\,  & \,  \frac{1}{t^{2}} \sum_{j< 0} 2^{-\frac12j} \norm{P_j\Big[\psi F(u(\tau))\Big]}_{L^1_tL^2_x([0,\delta_1/10]\times\R^5)}  \label{v-inf-duham-low}  \\
&\,+\,\frac1{t^2} \sum_{j\geq 0}2^{\frac14j} \norm{P_j\Big[\psi F(u(\tau))\Big]}_{L^{2}_tL^{\frac43}_x([0,\delta_1/10]\times\R^5)}\label{v-inf-duham-high}
\end{align}
For the low frequency part  \eqref{v-inf-duham-low},  
by Bernstein's estimates and H\"older's inequality,   we have 

\beq
 t^2  \eqref{v-inf-duham-low}
  \les    \sum_{j<0}2^{\frac{j}3} \norm{F(u)}_{L^1_tL^{\frac{3}{2}}_{x}([0,\delta_1/10]\times\R^5)} \les \delta_1^{\frac13 } \norm{u}^2_{L^3_{t,x} ([0,\delta_1/10 ]\times \R^5) }\les 1. 
\label{v-duhamel-Linf-3}
\eeq
For \eqref{v-inf-duham-high},
it suffices to  estimate 
\begin{align}
&   \sum_{j\geq 0} 2^{\frac14j}
\norm{\big[P_j, \psi \big]   F(u(\tau))}_{L^{2}_tL^{\frac43}_x([0,\delta_1/10]\times\R^5) }  \label{v-inf-high-1}  \\
+ & \sum_{j\geq 0} 2^{\frac14j}   \norm{ \psi P_j F(u(\tau))}_{L^{2}_tL^{\frac43}_x([0,\delta_1/10]\times\R^5)} . \label{v-inf-high-2}
\end{align}
For \eqref{v-inf-high-1}, by commutator estimates, the  Young inequality, the Sobolev embedding and Lemma \ref{local-1}, we have
\beq
\begin{split}
  \eqref{v-inf-high-1} \les  &  \sum_{j\geq 0} 2^{-\frac34j }\delta_1^{-1} \norm{Q_j F(u(\tau))}_{L^2_\tau L^{\frac43}_x([0,\delta_1/10]\times\R^5)}  \\
    \les &   \delta_1^{-1} \sum_{j\geq 0} 2^{-\frac34j} 2^{\frac14j} \norm{ F(u(\tau))}_{L^2_\tau L^{5/4}_x([0,\delta_1/10]\times\R^5)} \\
\les & \delta_1^{-\frac12}  \norm{u}^2_{L^\infty_tL^\frac52_x ([0,\delta_1/10]\times\R^5) }\les\delta_1^{-\frac12},
\end{split}\label{v-duhamel-high-1}
\eeq
where  $$Q_j f(x) = 2^{6j} \int_{\R^5} |y| \phi(2^jy)    |f|(x-y)  dy $$
and in the first inequality we used the mean value theorem. 
For \eqref{v-inf-high-2}, by the estimate \eqref{Pk Fu}, we have
\beq
\eqref{v-inf-high-2}  \,\les\,
 \sum_{j\geq 0} 2^{\frac14 j} \norm{P_j F(u) }_{L^2_\tau L^\frac43_x}\,\les\, \ep_0 .\label{v-duhamel-Linf-2}
\eeq
Hence, by \eqref{v-free-dispersive-infty}-\eqref{v-duhamel-Linf-2}, we  have
$\norm{v(t)}_{L^\infty_x(\R^5)} \les \delta_1^{-\frac12} t^{-2}. $

\vskip0.2cm
Now we  consider \eqref{v-stri}.
 For simplicity, we write 
 $$\norm{v}_{S(\R)} \,:=\,\norm{v}_{L^3_{t,x}(\R\times\R^5)}+\norm{v}_{L^\infty_t(\R  ; \dot B^{\frac12}_{2,1}(\R^5)) }.$$
 For the linear part, by the  Strichartz estimates and Lemma \ref{besov-a},  we have
\beq\label{v-free-L3}
\norm{S(t)(\psi u_0,\psi u_1)}_{S(\R)} 
\les   \norm{ \uui }_{\bb(\R^5) }\, \les\, 1.
\eeq
By the Strichartz estimates and repeating the arguments that deal with \eqref{v-inf-duhamel}, we have

\beq\nonumber
\begin{split}
        &\,\norm{\int_0^{\delta_1/10 } S(t-\tau) (0, \psi F(u(\tau))) d\tau}_{S(\R)}   \\
  \les & \,\norm{\int^{\delta_1/10}_0 \frac{\sin (\tau |\nabla|)}{|\nabla|} (\psi F(u(\tau))) d\tau }_{\dot B^{\frac12}_{2,1}(\R^5) } +\norm{\int^{\delta_1/10}_0 \frac{\cos (\tau |\nabla|)}{|\nabla|} (\psi F(u(\tau))) d\tau }_{\dot B^{\frac12}_{2,1}(\R^5) }\\
  \les\, &\delta_1^{-\frac12}.
\end{split}
\eeq
This completes the proof.
\end{proof}

\subsection{The energy part of the solution $u$}\label{energy part}

\begin{lemma}\label{w-energy}
We have

\vskip-0.2cm
\beq
\norm{\overrightarrow{w}(\delta_1)}_{\HL(\R^5)} \les \delta_1^{-1/2}.%
\eeq
\end{lemma}
\begin{proof}

By the definition of $w$, for $t\geq \delta_1$, we have
\beq\label{w}
\begin{split}
  w(t)=& S(t)\big((1-\psi)u_0,(1-\psi)u_1\big) - \int^{\frac{\delta_1}{10}}_0 S(t-\tau)(0,(1-\psi)F(u(\tau))) d\tau\\
      & -\int_\frac{\delta_1}{10}^t S(t-\tau )(0,F(u(\tau))) d\tau.
\end{split}
\eeq

First, we consider the contribution of  the third term of \eqref{w}. Taking $r=\frac{50}{13}$ in \eqref{sup-2r-5/r,r}, by interpolation, we have
\beq\label{L2L4}
\begin{split}
 \norm{u}^2_{L^2_tL^4_x([\frac{\delta_1}{10},\delta_1]\times \R^5) }\, &= \,\int_{\delta_1/10}^{\delta_1} \norm{u}_{L^{\frac{50}{13} }_x (\R^5)} \norm{u}_{L^{\frac{25}{6} }_x (\R^5)} dt \\
& \les\, \,\delta_1^{-\frac12} \sup_{t\in [\frac{\delta_1}{10},\delta_1]}\big[ t^{\frac{7}{10} }\norm{u(t)}_{L^{\frac{50}{13} }_x(\R^5) } \big] \norm{u}_{L^\frac54_tL^\frac{25}6_x([\frac{\delta_1}{10},\delta_1]\times \R^5) }\,\les \,  \delta_1^{-\frac12}.
\end{split}
\eeq
From this inequality and Strichartz, we have


\beq\begin{split}
&\Big\| \int^{\delta_1 }_{\delta_1/10} S(\delta_1-\tau) (0,F(u(\tau))) d\tau\Big\|_{\dot H^1_x(\R^5)} + \Big\| \pa_t \bigg[\int^{t }_{\delta_1/10} S (t-\tau) (0,F(u(\tau))) d\tau \bigg]_{t=\delta_1}\Big\|_{L^2_x(\R^5)}  \\
\les \,\,  &  \norm{ F(u) }_{L^1_\tau L^2_x([\frac{\delta_1}{10},\delta_1]\times\R^5)}
 \les \norm{u}^2_{L^{2}_\tau L^4_x([\frac{\delta_1}{10},\delta_1]\times\R^5)} \les \delta_1^{-\frac12}.
    \end{split}
\eeq
By Strichartz, radial Sobolev inequality \eqref{ra-sblv-infty-2} and H\"older, the second term of \eqref{w} can be estimated as
\beq
\begin{split}
       &\,\Big\|\int_0^{\delta_1/10} \vec{S}(t-\tau)\big(0,(1-\psi)F(u(\tau)) \big) d\tau\Big|_{t=\delta_1}\Big\|_{\HL(\R^5)}\\
\les\, &\, \norm{(1-\psi)F(u(\tau))}_{L^1_\tau L^2_x([0,\frac{\delta_1}{10}]\times\R^5)}\\
\les \, & \, \delta_1^{\frac12} \norm{(1-\psi) u}_{L^{\infty}_{t,x}([0,\frac{\delta_1}{10}]\times\R^5) }^{\frac12} \norm{u}^{\frac32}_{L^3_{t,x}([0,\frac{\delta_1}{10}]\times\R^5)}\\
 \les\, &\, \delta_1^{-\frac12} \norm{u}_{L^\infty_t \dot{B}^{1/2}_{2,1}}^{\frac12} \les \delta_1^{-1/2}.
\end{split}
\eeq
Hence, it remains  to estimate 
\beq
\norm{ (1-\psi)(u_0,u_1) }_{\HL(\R^5)}.
\eeq
%
For $u_0$, by radial Sobolev inequality \eqref{ra-sblv-infty-2} and polar coordinates,
\beq\label{nabla chi u_0}
  \norm{(\nabla \psi) u_0}^2_{L^2(\R^5)} \les\,  \delta_1^{-1} \int_{\delta_1/100}^{2\delta_1}\int_{\mathbb{S}^4} u^2_0(r) d\sigma(\omega)r^4 dr
    \,\les\,  \norm{u_0}_{\dot B^{\frac12}_{2,1}(\R^5)}^2 \,\les\,1 . %
\eeq
By the inequality \eqref{u_01-L^2}, we have
\beq\label{u_0}
\norm{(1-\psi)\pa_r u_0}_{L^2_x(\R^5)}\, \les\, \delta_1^{-\frac12}.
\eeq
For $u_1$, by the inequality \eqref{u_1--1} and polar coordinates, one can deduce that
\beq\label{u_1-2}
\norm{(1-\psi)u_1}_{L^2_x(\R^5)}^2 \,\les \,  \int^\infty_{\frac{\delta_1}{10}}  \int_{\mathbb{S}^4} |u_1(r)|^2 r^4 d\sigma(\omega) dr\, \les \, \int^\infty_{\frac{\delta_1}{10}} r^{-2} dr    \, \les \,  \delta_1^{-1}.
\eeq
This inequality together with \eqref{nabla chi u_0} and \eqref{u_0} implies that
\beq
\left\|(1-\psi)(u_0,u_1) \right\|_{\HL(\R^5)} \les \delta_1^{-\frac12}.
\eeq
This completes the proof.

\end{proof}%
%
%
%

\subsection{Global well-posedness}\label{global-wellposed}

In this subsection, we prove that the solution $u$  is globally well-posed.
We emphasize that  the constants in `` $\les$ '' in this subsection depend only upon  $\delta_1$ and $\norm{\uui}_{\BB(\R^5)} $.

\begin{theorem}\label{global-thm}
Let $u$ be the solution to \eqref{NLW} with initial data $\uui\in\BB$. Then $u$ is globally well-posed and such that for any compact interval $J\subset \R$,
\vskip -0.3cm
\beq\label{u-J}
\norm{u}_{L^3_{t,x}(J\times\R^5)}\,<\,C(J,\norm{\uui}_{\BB},\delta_1).
\eeq

\end{theorem}

\begin{proof}

By Lemma \ref{blow-criterion}, it suffices  to show \eqref{u-J}. 
And by the time reversibility of the wave equation, we just need consider the part of  $t\geq 0$.

For $t\geq\delta_1$, by $u(t)=w(t)+v(t)$ and the formula \eqref{v(t)} of $v(t)$, we have
\beq\label{w-equation}
w_{tt}-\Delta w=-|u|u.
\eeq
We define the energy of $w$ as \eqref{Energy-p} by
\beq
E(w(t))=\frac12\int_{\R^5}|\nabla_{t} w(t) |^2 dx +\frac12\int_{\R^5} |\nabla_{x} w(t) |^2 dx +\frac13\int_{\R^5} |w(t)|^3dx.
\eeq
By the estimates \eqref{u_H<B}, \eqref{v-stri},  Lemma \ref{w-energy} and  interpolation, we have 
\beq
E(w(\delta_1)) \les 1.
\eeq
Now, we consider
\beq\label{E(w)'}
\begin{split}
&|\frac{d}{dt} E(w(t))|\\
= \,&|\int_{\R^5} \big( |w|w-|u|u  \big) w_t dx| \\
                      \les\,& \norm{v}_{L^6_x(\R^5)} \norm{w_t}_{L^2_x(\R^5)} \norm{w}_{L^3_x(\R^5)} + \norm{w_t}_{L^2_x(\R^5)} \norm{v}_{L^6_x(\R^5)} \norm{v}_{L^3_x(\R^5)}.
\end{split}
\eeq
By interpolation and the dispersive estimate \eqref{v-dispersive} of $v$, we have 
\begin{align}
\norm{v}_{L^6_x(\R^5)} \norm{w_t}_{L^2_x(\R^5)} \norm{w}_{L^3(\R^5)}
&\les\,  \frac{1}{t} \norm{v}_{L^3_x(\R^5)}^{\frac12}  \norm{w_t}_{L^2_x(\R^5)} \norm{w}_{L^3_x(\R^5)}\nonumber \\
&\les \, \frac{1}{t} E(w(t))^{\frac56}\norm{v}_{L^3_x(\R^5)}^{\frac12} \les\, \frac{1}{t}[E(w)(t)+ \norm{v}^3_{L^3_x(\R^5)}] \label{wtvw},\\
\norm{w_t}_{L^2_x(\R^5)} \norm{v}_{L^6_x(\R^5)} \norm{v}_{L^3_x(\R^5)}  &\les\,   \frac{1}{t}  E(w(t))^{\frac12} \norm{v}_{L^3_x(\R^5)}^{\frac32}\les \frac{1}{t}[E(w)(t)+ \norm{v}^3_{L^3_x(\R^5)}] \label{wtvv}.
\end{align}
Substituting   \eqref{wtvw} and \eqref{wtvv} into \eqref{E(w)'}, we obtain
\beq
\frac{d}{dt} E(w(t)) \leq C\frac{1}{t} ( E(w(t)) + \norm{v}_{L^3_x(\R^5)}^3).
\eeq
Hence,
\beq
\frac{d}{dt}[t^{-C} (E(w(t)))] \leq t^{-C-1} \norm{v}_{L^3_x(\R^5)}^3.
\eeq
This estimate and the inequality \eqref{v-stri} imply that
\beq\label{w-energy-2}
E(w(t)) \leq C_1(1+|t|)^{C_2}.
\eeq
Thus, for any compact interval $J\subset \R$, we have
\beq
\norm{u}_{L^3_{t,x}(J\times\R^5)} \leq \norm{v}_{L^3_{t,x}(J\times\R^5)}+\norm{w}_{L^3_{t,x}(J\times\R^5)}<\infty.
\eeq

\end{proof}

\section{Scattering}\label{scattering}

In this section we prove  the Theorem \ref{main-theorem} by assuming Proposition  \ref{u-3-bounded}, that is, removing the dependence of $\delta_1$ in \eqref{u-scatter}.
%
From the arguments in
Section \ref{local-theory}, we have $\delta_1=\frac{\delta}{2R}$,
where $\delta$ and $R$ depending the scaling and spatial profile
of the initial data, respectively.
We give the heuristic idea
of the proof by analysing the effect of the parameters  $\delta$ and
 $R$ on the critical norm $L^3_{t,x}(\R\times\R^5)$.

 Note that the
critical norm $L^3_{t,x}(\R\times\R^5)$ of the solution to the
nonlinear equation \eqref{NLW} is invariant under the scaling
transform. Hence the parameter $\delta$
may not be the main difficulty to prove
Theorem \ref{main-theorem}.
On the other hand, the latter parameter  $R$ relies on the spatial profile of the initial data.
For example, let $R$ be  the parameter corresponding to the initial data $\uui$ with compact support.
The linear evolution $\vec S(t)$  for $t\in\R$ does not change the critical norm
$\HH$  by the Plancherel theorem, but owing to the Huygens principle for the odd-dimension
linear wave equations, it does change the  spatial support of the initial datum.
Thus, for  the initial data $\vec S(t_0)(u_0,u_1)$, the spatial
parameter(may be chosen as $R+t_0$) is likely very huge, when $t_0$ is large enough.
However, the $\BB$ norm may become huge under the evolution of $\vec S(t)$.
Indeed, if $\norm{\uui}_{\BB(\R^5)} =1$, then we have $\big\|\vec S(t_0)\uui\big\|_{\BB(\R^5)} \ra \infty $ as $t_0\ra\infty.$  \footnote{By interpolation and density, it suffices to show that $\lim\limits_{t\ra\infty}\big\|e^{it|\nabla|} f\big\|_{\dot B^{-2}_{\infty,\infty}(\R^5)}= 0,$ for $f \in \mathcal{S}(\R^5),$  which is a direct consequence of the dispersive estimate
\eqref{Dsp-est} and Bernstein's estimates.} Hence,
if $\big\|\vec S(t_0)\uui\big\|_{\BB(\R^5)} \les 1$,
then $t_0$ remains bounded.  Taking account of this fact, one may conquer the difficulties caused by the parameter $R$.


\vskip0.2cm
To finish the proof of Theorem \ref{main-theorem}, we need the following theorem of profile decomposition.
%
%
%
\subsection{Profile decomposition}

Now, we recall the linear profile decomposition from \cite{ramos2012a} in the radial case.
We refer to \cite{BG-1999-AJM} for 
the profile decompositions in the  energy critical spaces. 

\begin{theorem}[Profile decomposition]\label{profile}
Let $C>0$ be a fixed number and
let $\uuin_n$ be a sequence in $\HH(\R^5)$, with
\beq
\norm{\uuin}_{\HH(\R^5)}\leq C.
\eeq
Then there exist a subsequence of $\uuin$ (which we still denote  by  $\uuin$), a sequence
$\ppj _{j\in \mathbb{N}}\subset\HH(\R^5)$, a sequence $ (R^N_{0,n},R^N_{1,n})_{N\geq1} \subset\HH(\R^5)$
and a sequence of  parameters
$(t_j^n,\lambda^n_j)\subset\R\times(0,\infty)$ such that for each $N\geq1$,
\beq\label{li-profile-1}
S(t)\uuin= \sum_{j=1}^N (\lambda^n_j)^2 \big[S(\lambda^n_j (t-t^n_j)) (\phi_0^j,\phi^j_1)\big](\lambda^n_j x) +S(t)(R_{0,n}^N,R_{1,n}^N)
\eeq
with
\beq\label{li-profile-error}
\lim_{N\ra\infty} \limsup_{n\ra\infty} \norm{S(t)(R_{0,n}^N,R_{1,n}^N) }_{L^3_{t,x}(\R\times\R^5)}=0.
\eeq
In addition,   the parameters $(t_j^n,\lambda^n_j)$ satisfy  the  orthogonality property:
for any $j\neq k$,
\beq\label{orthogonal-parameter}
\lim_{n\ra\infty} \left(\frac{\lambda_j^n}{\lambda_k^n} + \frac{\lambda_k^n}{\lambda_j^n} + (\lambda_k^n )^\frac12 (\lambda_j^n )^\frac12|t^n_j-t^n_k|\right)\,=\,\infty.
\eeq
Furthermore, for every $N\geq 1$,
\beq\label{li-profile-sobolev}
\norm{\uuin}^2_{\HH(\R^5)} =\sum_{j=1}^N \norm{\ppj}^2_{\HH(\R^5)}+\norm{(R^N_{0,n},R^N_{1,n})}^2_{\HH(\R^5)}+o_n(1).
\eeq
\end{theorem}
%

\subsection{End the proof of the main theorem}
Now, we apply the strategy in \cite{dodson-2016} to finish the proof of Theorem \ref{main-theorem}, that is,
remove the $\delta_1$  in Proposition \ref{u-3-bounded}.

We prove the Theorem \ref{main-theorem} by contradiction.
We assume $u$ is the solution to \eqref{NLW} with the initial data such that $\uui\in \BB(\R^5).$
For $M\geq0$, let
\beq
f(M)=\sup\{ \norm{u}_{L^3_{t,x} (\R\times \R^5)}:  \norm{\uui}_{\BB(\R^5)}\leq M  \}.
\eeq
Then by Proposition \ref{u-3-bounded}, $f$ is well defined.  In addition,  by Bernstein and Theorem \ref{lwp}, $f(M)<\infty$ when $M$
is small enough. From  the definition,  $f(M)$ is non-decreasing as $M$ increases.

If Theorem \ref{main-theorem} fails, then there exist $M_0<\infty$ and a sequence $\{\uuin\}_{n\in \mathbb{N} }\subset\BB(\R^5)$, such that
\beq\label{uB<C}
\norm{\uuin}_{\BB(\R^5)} \leq M_0,
\eeq
and the solution $u^n$ to \eqref{NLW} with  $(u^n(0),(\pa_tu^n)(0))=\uuin$ satisfying
\beq\label{u^n--infty}
\norm{u^n}_{L^3_{t,x}(\R\times\R^5)} \ra \infty,
\eeq
as $n\ra\infty.$
By Theorem \ref{profile},  we have
\beq\label{linear-u-phi}
S(t)\uuin= \sum_{j=1}^N (\lambda^n_j)^2 \big[S(\lambda^n_j (t-t^n_j)) (\phi_0^j,\phi^j_1)\big](\lambda^n_j x) +S(t)(R_{0,n}^N,R_{1,n}^N).
\eeq

In the proof of Theorem \ref{profile}, the author in \cite{ramos2012a} actually proved that
\beq\label{weak-lim}
\vec{S}(t+ t^n_j\lambda^n_j)\left[ (\lambda^n_j)^{-2} u^n_0\Big(\frac{\cdot}{\lambda^n_j}\Big),(\lambda^n_j)^{-3}u^n_1\Big(\frac{\cdot}{\lambda^n_j}\Big) \right] \Bigg{|}_{t=0} \rightharpoonup (\phi_0^j,\phi_1^j),
 \eeq
weakly in $\HH(\R^5)$ as $n\ra \infty.$ From this we can prove the following proposition.
\begin{proposition}\label{tnjlambda}
For fixed $j\in\mathbb{N}$, if $(\phi_0^j, \phi_1^j)\neq 0,$ then $|t^n_j\lambda^n_j|$ is bounded as $n\ra\infty$. 
\end{proposition}
\begin{proof}
First, if $t^n_j\lambda^n_j$ is unbounded for $n\in \mathbb{N}$, then  by taking a subsequence of $n$(still denoted by $n$), we assume that $|t^n_j\lambda^n_j| \ra \infty$, as $n\ra \infty$ .
In light of the heuristic analysis at the beginning of this section,
we may have
\beq\label{}
\vec S(t^n_j\lambda^n_j)\left[ (\lambda^n_j)^{-2} u^n_0\Big(\frac{\cdot}{\lambda^n_j}\Big),(\lambda^n_j)^{-3}u^n_1\Big(\frac{\cdot}{\lambda^n_j}\Big) \right]\, \ra\, (0,0).
\eeq
in $L^3_x\times W_x^{-1,3}(\R^5)$.
%
In fact,
using \eqref{uB<C} and the estimate \eqref{2r-5/r,r} in Section \ref{local-theory}, we have
\beq\label{stnj}
\begin{split}
      &\,\Big\|S(t^n_j\lambda^n_j)\Big[(\lambda^n_j)^{-2} u^n_0(\frac{\cdot}{\lambda^n_j}),(\lambda^n_j)^{-3} u^n_1(\frac{\cdot}{\lambda^n_j})  \Big]  \Big\|_{L^3_{x}(\R^5)}\\
  \les&\,  |t^n_j\lambda^n_j|^{-\frac13} \Big\| \Big((\lambda^n_j)^{-2} u^n_0(\frac{\cdot}{\lambda^n_j}),(\lambda^n_j)^{-3} u^n_1(\frac{\cdot}{\lambda^n_j}) \Big)\Big\|_{\BB(\R^5)} \\
  \les&\, |t^n_j\lambda^n_j|^{-\frac13}\norm{\uuin}_{\BB(\R^5)}\ra 0,
\end{split}
\eeq
as $n\ra\infty$.
Similarly,  by the dispersive estimate \eqref{Dsp-est}, Bernstein and interpolation, we have
\beq\label{dstnj}
\begin{split}
       &\,\Big\| \dot S(t^n_j\lambda^n_j)\Big[(\lambda^n_j)^{-2} u^n_0(\frac{\cdot}{\lambda^n_j}),(\lambda^n_j)^{-3} u^n_1(\frac{\cdot}{\lambda^n_j})  \Big]  \Big\|_{\dot W^{-1,3}_{x}(\R^5)}\\
  \les&\, |t^n_j\lambda^n_j|^{-\frac13} \norm{\uuin}_{\BB(\R^5)} \ra 0,
\end{split}
\eeq
as $n\ra\infty$.
Hence,  from the weak convergence relation \eqref{weak-lim},   \eqref{stnj}  and \eqref{dstnj} imply that  $(\phi_0^j, \phi_1^j)= (0,0)$.

\end{proof}

For simplicity, we assume that every  $(\phi_0^j,\phi_1^j)$ in \eqref{linear-u-phi}
 is nontrivial. By Proposition \ref{tnjlambda}, for each fixed $j$,  $t^n_j\lambda^n_j$
 is bounded,  and therefore after taking   a suitable subsequence of $n$ (still denoted by $n$),
 we can assume $t^n_j\lambda^n_j\ra t_j\in \R$ as $n\ra\infty$. Hence, if we denote $(\vp_0^j, \vp^j_1)= \vec{S}(-t_j)(\phi_0^j,\phi^j_1),$ then

\beq\label{phi-varphi}
\vec S(-\lambda^n_jt^n_j)(\phi^j_0,\phi^j_1)-(\vp^j_0,\vp^j_1)\ra 0,
\eeq
in $\HH(\R^5)$ as $n\ra\infty.$
Let
\begin{equation}\label{R-TR}
\left\{  \begin{aligned}
\tilde{R}_{0,n}^N &=R_{0,n}^N +
 \sum_{j=1}^N (\lambda^n_j)^2\big[S(-\lambda^n_jt^n_j)(\phi^j_0,\phi^j_1)\big](\lambda^n_jx) - (\lambda^n_j)^2 (\vp^j_0,\vp^j_1)(\lambda^n_jx),\\
\tilde{R}_{1,n}^N&= R_{1,n}^N + \sum_{j=1}^N (\lambda^n_j)^3 \big[ \dot S(-\lambda^n_jt^n_j) (\phi^j_0,\phi^j_1)\big](\lambda^n_jx) -(\lambda^n_j)^3\big[  (\vp^j_0,\vp^j_1)\big[ (\lambda^n_jx),
\end{aligned} \right.
\end{equation}
then
\beq\label{li-profile-error-tR}
\lim_{N\ra\infty} \limsup_{n\ra\infty} \norm{S(t)(\tilde{R}_{0,n}^N, \tilde{R}_{1,n}^N) }_{L^3_{t,x}(\R\times\R^5)}=0.
\eeq
%
%
Taking $t=0$ in \eqref{linear-u-phi},  by  \eqref{phi-varphi} and \eqref{R-TR},  we have
\beq\label{pertur-initial}
\uuin=\sum_{j=1}^N  \big(  (\lambda^n_j)^2 (\vp^j_0(\lambda^n_j x) , (\lambda^n_j)^3 \vp^j_1(\lambda^n_j x) \big) +  (\tilde{R}_{0,n}^N, \tilde{R}_{1,n}^N ).
\eeq
In addition, by the orthogonality  \eqref{orthogonal-parameter}  and Proposition \ref{tnjlambda}, we have for each $j\neq k$
\beq\label{orthogonal}
\lim_{n\ra \infty} \frac{\lambda_j^n}{\lambda_k^n} + \frac{\lambda_k^n}{\lambda_j^n}\, = \,\infty,
\eeq
as $n\ra\infty$. Thus, for fixed $j\in \mathbb{N}$, we have
\beq
\Big( ( \lambda^n_j)^{-2} u^n_0(\frac{x}{\lambda^n_j }) , (\lambda^n_j)^{-3} u^n_1(\frac{x}{\lambda^n_j }) \Big)
  \rightharpoonup (\vp^j_0,\vp^j_1)
\eeq
weakly in $\HH(\R^5)$ as $n\ra \infty$.
By Fatou's  lemma,
this fact and the inequality \eqref{uB<C}
imply
\beq
\norm{(\vp^j_0,\vp^j_1)}_{\BB(\R^5)}\leq M_0.
\eeq
On the other hand,  
\eqref{li-profile-sobolev} and \eqref{phi-varphi}  yield that
\beq
\sum_{j\geq1} \norm{(\vp^j_0,\vp^j_1)}^2_{\HH(\R^5)} \les \sup_{n\geq 1} \norm{\uuin}_{\HH(\R^5)}^2\les C^2_0.
\eeq
Hence, for fixed $\ep>0$, there exists a finite integer  $N_0$ such that
\beq\label{j-large-HH}
\sum_{j\geq N_0+1} \big\|(\vp^j_0,\vp^j_1)\big\|^2_{\HH(\R^5)} \leq \,\ep.
\eeq
By the local well-posedness theory, if  $\ep>0$ is small enough, then the solution $v^j$ to \eqref{NLW} with the initial data $( \vp^j_0,\vp^j_1)$ is globally well-posed and 
\beq\label{j-large-L3}
\norm{v^j}_{L^3_{t,x}(\R\times\R^5)}\les\, \big\|(\vp^j_0,\vp^j_1)\big\|_{\HH(\R^5)}, \text{\, for every\, } j\geq N_0+1.
\eeq

For $1\leq j\leq N_0$, as a consequence of Proposition  \ref{u-3-bounded}, the solution to \eqref{NLW} with the initial data $( \vp^j_0,\vp^j_1)$ is
globally well-posed and such that
\beq\label{j-small-L3}
\norm{v^j}_{L^3_{t,x}(\R\times\R^5)}\les_{M_0,j} 1.
\eeq
By the orthogonality property  \eqref{orthogonal}, for any $j\neq k$
 \beq\label{}
 \lim_{n\ra\infty } \iint_{ \R\times \R^5}   |(\lambda^n_j)^2 v^j(\lambda^n_j t,\lambda^n_j x) |^2 | (\lambda^n_k)^2 v^k(\lambda^n_k t,\lambda^n_k x) |dx dt =0.
\eeq
This together with the estimates \eqref{j-large-HH}, \eqref{j-large-L3} and \eqref{j-small-L3} implies
\beq\label{perturb-1}
\sup\limits_{N\geq N_0+1} \lim\limits_{n\ra\infty} \Big\|\sum_{1\leq j\leq N} (\lambda^n_j)^2 v^j(\lambda^n_j t,\lambda^n_j x) \Big\|_{L^3_{t,x}(\R\times\R^5)}
\eeq
is  bounded. 
Similarly, as a consequence of  the trivial estimate
\beq\nonumber
\begin{split}
    & \left| F(\sum_{j=1}^N (\lambda^n_j)^2 v^j(\lambda^n_j t,\lambda^n_j x) )-
                      \sum_{j=1}^N F( (\lambda^n_j)^2 v^j(\lambda^n_j t,\lambda^n_j x)) \right|\\
\les\,\, & \sum_{1\leq j, k\leq N, j\neq k }  |(\lambda^n_j)^2 v^j(\lambda^n_j t,\lambda^n_j x) )|   (\lambda^n_k)^2 v^k(\lambda^n_k t,\lambda^n_k x)|
\end{split}
\eeq
and the orthogonality property \eqref{orthogonal}, we have
\beq\label{perturb-2}
\lim_{n\ra\infty}\biggl\|F(\sum_{j=1}^N (\lambda^n_j)^2 v^j(\lambda^n_j t,\lambda^n_j x) )-
                      \sum_{j=1}^N F( (\lambda^n_j)^2 v^j(\lambda^n_j t,\lambda^n_j x))\biggr\|_{L^\frac{3}{2}_{t,x}(\R\times\R^5)}=0.
\eeq

Let $u^n_N$ be an  approximate  solution to \eqref{NLW} defined by
$$ u^n_N = \,\sum_{j=1}^N (\lambda^n_j)^2 v^j (\lambda^n_j x)+ S(t)\RR.  $$
Then,
recall the property \eqref{li-profile-error-tR} for $\RR$ and
the fact that \eqref{perturb-1} is uniformly bounded for $N\geq N_0+1$,  we obtain
\beq\label{perturb-1-}
\limsup_{N\ra\infty}\lim\limits_{n\ra\infty} \norm{u^n_N}_{L^3_{t,x}(\R\times\R^5)} \les 1.
\eeq
Moreover, 
combining
\eqref{perturb-2}, the property \eqref{li-profile-error-tR} for $\RR$,  and H\"older's inequality, we have

\beq\label{perturb-2-}
\limsup\limits_{N\ra\infty}\lim\limits_{n\ra\infty}
\Big\|
F(u^n_N)-\sum_{j=1}^N F\big((\lambda^n_j)^2v^j(\lambda^n_j t, \lambda^n_j x)
     \big) \Big\|_{L^\frac32_{t,x}(\R\times\R^5)} =0.
\eeq
%
%

Utilizing  Theorem \ref{perturbation}, by \eqref{pertur-initial}, \eqref{perturb-1-} and  \eqref{perturb-2-},  we have that for $n$ large enough, the solution $u^n$ to the equation \eqref{NLW}  with initial data $\uuin$ is global and such that
\beq
\lim_{n\ra \infty} \norm{u^n}_{L^3_{t,x}(\R\times \R^5)}
\eeq 
 is bounded, which contradicts with
 the hypothesis  \eqref{u^n--infty} of $u^n$. Thus, we have proved Theorem \ref{main-theorem}.

%
\section{Hyperbolic coordinates and spacetime estimates} \label{hyper-scatter}

In this section, we will finish the proof of  Proposition \ref{u-3-bounded}. 
We first reduce Proposition  \ref{u-3-bounded} to estimating the $L^3_{t,x}$ norm of $w$ on
the region $\Omega_2$, which will be defined below. Without loss of generality, we assume that $\delta_1<\frac14$.
And as in Subsection \ref{global-thm}, we also note the constants in ``$\les$'' in the following of this section may be different in each step and  are  dependent on  $\delta_1$ and $\norm{\uui}_{\BB(\R^5)}$.

\subsection{Reduction of  the proof of Proposition \ref{u-3-bounded} }\label{reduction}
Now we consider the $L^3_{t,x}$ norm of $u $ on the $\R_+\times \R^5$.
First, we  split time-spatial region $\R_+\times \R^5$ as the union
\beq
\R_+\times \R^5=\, \Omega_1 \cup \Omega_2 \cup \Omega_3,
\eeq
where
$ \Omega_1=\,\{(t,x)\in\R_+\times \R^5:\, |x|\geq t+\frac12    \}$ and
$\Omega_2=\, \{(t,x)\in\R_+\times \R^5:\, (t+(1-\delta_1))^2-|x|^2\geq 1    \}.$
Since $\delta_1<\frac14$, there exists a large constant $C>0$, such that
$$ \Omega_3 \subset \{(t,x)\in\R_+\times \R^5:\, t+  |x| \leq C   \}.$$
Recalling the estimate \eqref{u-out-cone} in Section \ref{local-theory}, we obtain
$\norm{u}_{L^3_{t,x}(\Omega_1)}\les 1.$
For the bounded region $\Omega_3$,  Theorem \ref{global-thm} yields
$\norm{u}_{L^3_{t,x}(\Omega_3)}\les 1.$ Hence, we just need to consider the $L^3_{t,x}$ norm of $u$ on the region $\Omega_2$.
By the estimate \eqref{v-stri} for $v$, we are reduced to show  $\norm{w}_{L^3_{t,x}(\Omega_2)}\les 1.$

%

\subsection{Hyperbolic coordinates}

For the radial solution $u(t,x)$  to \eqref{NLW}, if we denote $u(t,r)=u(t,x)$ for $r=|x|$, then
\beq
\pa_{tt}(r^2u) -\pa_{rr}(r^2u)=-2u-r^2 |u|u.
\eeq

Denote $\bu(t,r) =u(t-(1-\delta_1),r)$ and $\bv,\bw$ similarly.
Let $(t,r)=(e^\tau\cosh s, e^\tau\sinh s) $, then $drdt=e^{2\tau} d\tau d s$.
We denote the hyperbolic transforms by
\beq\label{hy-trans-u}
\tilde u=\frac{e^{2\tau} \sinh^2 s }{s^2 } \bu(e^\tau\cosh s, e^\tau\sinh s),
\eeq
\beq\label{hy-trans-v}
\tilde v=\frac{e^{2\tau} \sinh^2 s }{s^2 } \bv(e^\tau\cosh s, e^\tau\sinh s),
\eeq
and
\beq\label{hy-trans-w}
\tilde w=\frac{e^{2\tau} \sinh^2 s }{s^2 } \bw(e^\tau\cosh s, e^\tau\sinh s).
\eeq
Hence, we have
\beq\label{tu-wave}
\pa_{\tau\tau}(s^2\tu)-\pa_{ss}(s^2\tu)=-\frac{2 s^2}{\sinh^2 s} \tu-\frac{s^4}{ \sinh^2s  } |\tu|\tu,
\eeq
\beq
\pa_{\tau\tau}(s^2\tv)-\pa_{ss}(s^2\tv)=-\frac{2 s^2}{\sinh^2 s} \tv,
\eeq
\beq\label{tw-wave}
\pa_{\tau\tau}(s^2\tw)-\pa_{ss}(s^2\tw)=-\frac{2 s^2}{\sinh^2 s} \tw-\frac{s^4}{ \sinh^2s  } |\tu|\tu.
\eeq
Define the hyperbolic energy of $\tw$ by
\beq
E_h (\tw)(\tau)= \int_0^\infty \big[\frac12|(s^2\tw)_\tau|^2 +\frac12 |(s^2\tw)_s|^2 +  \frac{|s^2\tw|^2 }{\sinh^2 s} + \frac13 \frac{|s^2 \tw|^3}{ \sinh^2 s} \big]ds.
\eeq

\subsection{The hyperbolic energy for some $\tau_0\geq 0$}

First, we want to prove $E_h (\tw)(\tau)$  is bounded for some $\tau_0\geq0$.
We claim that it suffices to show the boundedness of
\beq
\int_0^\infty \big[|(s^2\tw)_\tau(\tau_0,s)|^2 + |(s^2\tw)_s(\tau_0,s)|^2\big]ds
\eeq
 for some $\tau_0>0$.

 To prove this claim, we need   the sharp Hardy inequality,
\beq\label{sharp-hardy}
\left(\frac{d-2}{2} \right)^2 \int_{\R^d} \frac{|f(x)|^2}{|x|^2} dx \leq \int_{\R^d} |\nabla f|^2(x) dx .
\eeq
By polar coordinates, we rewrite this inequality for radial functions,
\beq\label{sharp-hardy-ra}
\left(\frac{d-2}{2} \right)^2 \int_0^\infty  |f(r)|^2 r^{d-3} dr \, \leq\,   \int_0^\infty   |\pa_r f(r)|^2(r)r^{d-1} dr.
\eeq
Then, this inequality and integration by parts imply that
\begin{align}
\int_0^\infty |(s^2\tw(\tau_0))_s|^2 ds  =\,\,& \int_0^\infty  s^4 \tw_s^2(\tau_0) ds + 4 \int_0^\infty  s^2\tw(\tau_0) s\tw_s(\tau_0) ds +4\int_0^\infty  s^2 \tw^2(\tau_0) ds \nonumber \\
       =\,\, & \int_0^\infty  s^4 \tw_s^2(\tau_0) ds -2 \int_0^\infty  s^2 \tw(\tau_0) ^2ds\,\, \geq \,\, \frac19 \int_0^\infty  s^4 \tw_s^2(\tau_0) ds.
     \end{align}
In addition, by  H\"older and Sobolev in polar coordinates, we have
\beq
\int_0^\infty \frac{|s^2 \tw(\tau_0)|^3}{\sinh^2 s} ds =\int_0^\infty \frac{s^2}{\sinh^2 s } |\tw(\tau_0)|^3 s^4 ds \les (\int_0^\infty |\tw(\tau_0)|^{\frac{10}3 } s^4 ds) ^\frac9 {10}\les (\int_0^\infty |\tw_s (\tau_0)|^2 s^4 ds )^\frac32.
\eeq
By Hardy's inequality,   we have  
\beq
\int_0^\infty \frac{|s^2 \tw(\tau_0)|^2}{\sinh^2 s} ds \les\, \int_0^\infty \frac{1}{s^2} |\tw(\tau_0)|^2  s^4 ds \les\, \int_0^\infty |\tw_s(\tau_0)|^2 s^4 ds . 
\eeq
Hence the claim follows.

\subsubsection{\textbf{The hyperbolic energy for $s>s_0>0$ }} 
For $\tau\in[0,1]$ and  sufficiently large $s_0>0$, we can assume that $e^{\tau-s_0} <\frac12-\delta_1$.
By the finite  speed of propagation,
$v(t,r)$ are supported in the region
$ \{(t,r)\in\R\times \R_+:    r-t \, \les \, \delta_1/5\} .$
Then,  for $\tau\in[0,1]$ and $s\geq s_0$, we have
 $$e^\tau\sinh s-[e^\tau\cosh s-(1-\delta_1)]=1-\delta_1-e^{\tau-s} >\frac12>\frac{\delta_1}{5},$$
 which 
leads to
$v(e^\tau\cosh s-(1-\delta_1),e^\tau\sinh s)=0.$
Hence, for $\tau\in[0,1]$, we have
\beq\label{u-energy-part1}
\int_{s_0}^\infty \frac12|(s^2\tw)_\tau(\tau,s)|^2 +\frac12 |(s^2\tw)_s(\tau,s)|^2 \  ds =\int_{s_0}^\infty \frac12|(s^2\tu)_\tau(\tau,s)|^2 +\frac12 |(s^2\tu)_s(\tau,s)|^2 \ ds.
\eeq
Since $u$ is a radial solution to the equation \eqref{NLW}, we have, by \eqref{SL-formula-radial},
\beq
\begin{split}
  r^2u(t,r)\,=\,\,&\frac12[(r-t)^2u_0(r-t)+(r+t)^2u_0(r+t)]\\
        & -\frac12 t r^{-1}\int^{r+t}_{r-t} s u_0(s) ds\\
        & +\frac1{4r} \int^{r+t}_{r-t} s  (s^2+r^2-t^2) u_1(s) ds\\
        & + \frac1{4r} \int_0^t \int^{r+t-s}_{r-t+s} \rho  (\rho^2+r^2-(t-s)^2)|u|u(s,\rho) d\rho ds,
\end{split}
\eeq
for $r\geq t\geq0.$
Hence, by the hyperbolic transform \eqref{hy-trans-u}, we have

\begin{align}
   &s^2\tu(\tau,s) \nonumber\\
  =\,\,&\frac12[((1-\delta_1-e^{\tau-s})^2u_0(1-\delta_1-e^{\tau-s})+(e^{\tau+s} -(1-\delta_1) )^2u_0(e^{\tau+s} -(1-\delta_1))] \label{u-0-1}\\
        & -\frac12 (e^\tau\cosh s-(1-\delta_1)) (e^\tau\sinh s)^{-1}\int^{e^{\tau+s} -(1-\delta_1)}_{1-\delta_1-e^{\tau-s}}  \rho u_0(\rho) d\rho \label{u-0-2}\\
        & +\frac14 \int^{e^{\tau+s} -(1-\delta_1)}_{1-\delta_1-e^{\tau-s}} \rho  \frac{\rho^2+(e^{\tau+s} -(1-\delta_1))( 1-\delta_1-e^{\tau-s})}{e^{\tau}\sinh s } u_1(\rho) d\rho\label{u-1}\\
        & + \frac14 \int_{1-\delta_1}^{e^\tau\cosh s} \int^{ e^{\tau+s}-t}_{t-e^{\tau-s} }  \rho  \frac{\rho^2+(e^{\tau+s} -t)(t-e^{\tau-s})}{e^{\tau}\sinh s} |\bu|\bu(t,\rho) d\rho dt \label{Duhamel}.
\end{align}

For \eqref{u-0-1}, by a direct calculation, we  obtain

\beq
(\pa_{\tau}+\pa_{s} )\eqref{u-0-1}\,=\, 2(e^{\tau+s} -(1-\delta_1) ) e^{\tau+s} u_0(e^{\tau+s} -(1-\delta_1))+ (e^{\tau+s} -(1-\delta_1) )^2 u'_0(e^{\tau+s}   -(1-\delta_1)) e^{\tau+s},
\eeq
\beq
(\pa_{\tau}-\pa_{s} )\eqref{u-0-1}\,=\, 2(1-\delta_1-e^{\tau-s})e^{\tau-s}u_0(1-\delta_1-e^{\tau-s}) + ((1-\delta_1-e^{\tau-s})^2u'_0(1-\delta_1-e^{\tau-s})e^{\tau-s}.
\eeq
Using the estimate \eqref{u_01-L^2} in Section \ref{basics} and polar coordinates, we deduce that
\beq\label{u-0-1-tau+s}
\begin{split}
  \int_{s_0} \left|(e^{\tau+s} -(1-\delta_1) ) e^{\tau+s} u_0(e^{\tau+s} -(1-\delta_1)) \right|^2ds  \les \int_0^\infty |u_0(r)|^2 r^3 dr\,\les\, 1,
\end{split}
\eeq
\beq
\begin{split}
  \int_{s_0} \left|(e^{\tau+s} -(1-\delta_1) ) ^2e^{\tau+s} u'_0(e^{\tau+s} -(1-\delta_1)) \right|^2ds  \les \int_0^\infty |\pa_r u_0(r)|^2 r^5 dr\,\les\, 1.
\end{split}
\eeq
By the radial Sobolev inequality \eqref{ra-sblv-infty-2},
we have $|u_0(r)|\les r^{-2}$. This estimate and the inequality \eqref{rad-sobolev-u_0}
imply
\beq
\begin{split}
&  \int_{s_0}^\infty  |(1-\delta_1-e^{\tau-s})e^{\tau-s}u_0(1-\delta_1-e^{\tau-s})|^2 ds  +\int_{s_0}^\infty | ((1-\delta_1-e^{\tau-s})^2u'_0(1-\delta_1-e^{\tau-s})e^{\tau-s}|^2 ds\\
\les &    \int_{s_0}^\infty e^{-2s} ds\les 1.
 \end{split}
\eeq
We now take the derivatives of \eqref{u-0-2} with respect to  $\tau $ and $s$, then
\begin{equation}
\pa_\tau \eqref{u-0-2}\,=\,  \frac{1-\delta_1}{2e^\tau\sinh s} \int^{e^{\tau+s} -(1-\delta_1)}_{1-\delta_1-e^{\tau-s}} \rho u_0(\rho) d\rho+ I_1 +I_2 \label{u-0-2-tau}
\end{equation}
and
\beq
\pa_s\eqref{u-0-2}\,= \,\pa_s\left(\frac{e^{\tau}\cosh s-(1-\delta_1) }{2e^\tau\sinh s}\right) \int^{e^{\tau+s} -(1-\delta_1)}_{1-\delta_1-e^{\tau-s}} \rho u_0(\rho) d\rho+ I_1 -I_2, \label{u-0-2-s}
\eeq
where
\beq
I_1=\frac{e^{\tau} \cosh s-(1-\delta_1)}{2e^{\tau}\sinh s} e^{\tau+s} (e^{\tau+s}-(1-\delta_1)) u_0(e^{\tau+s}-(1-\delta_1)),
\eeq
\beq
I_2=\frac{e^{\tau} \cosh s-(1-\delta_1)}{2e^{\tau}\sinh s} e^{\tau-s} (1-\delta_1-e^{\tau-s}) u_0(1-\delta_1-e^{\tau-s}).
\eeq
For the first term in the right hand of \eqref{u-0-2-tau}, by the inequality \eqref{rad-sobolev-u_0}, we have

\beq
\int_{s_0}^\infty \bigg| e^{-s}  \int^{e^{\tau+s} -(1-\delta_1)}_{1-\delta_1-e^{\tau-s}} \rho u_0(\rho) d\rho \bigg|^2 ds \les \int_{s_0}^\infty  e^{-2s} ds\les 1.
\eeq
By similar estimates, one can  find that the contribution of the first term in the right hand of \eqref{u-0-2-s} to \eqref{u-energy-part1} is finite.
For $I_1$, changing of variables and the inequality \eqref{u_01-L^2} yield
\beq
\int_{s_0}^\infty |I_1|^2 ds\les \int_{s_0}^\infty |e^{2s}u_0(e^{\tau+s}-(1-\delta_1))|^2ds\,
\les\, \int_{\frac12e^{s_0}}^\infty \rho^3 |u_0(\rho)|^2d\rho\les 1.
\eeq
For $I_2$, by  
\eqref{ra-sblv-infty-2}, 
we obtain
\beq
\int_{s_0}^\infty |I_2| ds \les \int_{s_0}^\infty  |e^{-s} |^2 ds\les 1.
\eeq

Next, we consider the contribution of \eqref{u-1} to \eqref{u-energy-part1}.
For simplicity, we consider
\begin{align}
\frac12( \pa_{\tau} -\pa_{s} ) \eqref{u-1}\, =\,\,  & e^{\tau-s}  (1-\delta_1-e^{\tau-s})^2  u_1(1-\delta_1-e^{\tau-s}) \label{u1-tau-s-1}\\
   &   +\frac{e^{\tau-s} (e^{\tau+s} -(1-\delta_1))^2}{(e^{\tau+s} -e^{\tau-s})^2  } \int^{e^{\tau+s} -(1-\delta_1)}_{1-\delta_1-e^{\tau-s}}   \rho u_1(\rho) d\rho \label{u1-tau-s-2} \\
   &   +\frac{e^{\tau-s}}{(e^{\tau+s}-e^{\tau-s})^2}\int^{e^{\tau+s} -(1-\delta_1)}_{1-\delta_1-e^{\tau-s}} \rho^3 u_1( \rho) d\rho,  \label{u1-tau-s-3}
\end{align}
and
\begin{align}
\frac12( \pa_{\tau} +\pa_{s} ) \eqref{u-1}\, =\, \, & e^{\tau+s}  (e^{\tau+s} -(1-\delta_1))^2  u_1(e^{\tau+s} -(1-\delta_1)) \label{u1-tau+s-1}\\
   &   +\frac{e^{\tau+s} (1-\delta_1-e^{\tau-s})^2}{(e^{\tau+s} -e^{\tau-s})^2  } \int^{e^{\tau+s} -(1-\delta_1)}_{1-\delta_1-e^{\tau-s}}   \rho u_1(\rho) d\rho \label{u1-tau+s-2} \\
   &   +\frac{e^{\tau+s}}{(e^{\tau+s}-e^{\tau-s})^2}\int^{e^{\tau+s} -(1-\delta_1)}_{1-\delta_1-e^{\tau-s}} \rho^3 u_1( \rho) d\rho.  \label{u1-tau+s-3}
\end{align}
Using the estimates \eqref{rad-sobolev-u_1} and \eqref{u_01-L^2},
we can easily estimate the contributions of   \eqref{u1-tau-s-1}-\eqref{u1-tau+s-2}
to \eqref{u-energy-part1}. 
Let ${\mathbb I}_{J} (y)$ be the characteristic function of an interval $J\subset \R$. 
For \eqref{u1-tau+s-3}, by the inequality \eqref{u_01-L^2}  and a change of variables, we see that
\begin{align}
 \int_{s_0}^\infty \left| e^{-s}\int^{e^{\tau+s} -(1-\delta_1)}_{1-\delta_1-e^{\tau-s}} \rho^3 u_1( \rho) d\rho \right|^2  ds
\, \les \  & \int_{s_0}^\infty \left| e^{-s} \int^{2e^s}_{0} \rho^3 |u_1( \rho)| d\rho \right|^2  ds  \nonumber \\
 \les \ &   \int_0^\infty \left|\frac{1}{\eta} \int^{2\eta}_{0} \rho^3 |u_1( \rho)| d\rho \right|^2  \frac1\eta d\eta \nonumber \\
 \les \ & \int_0^\infty \left|\frac{1}{\eta} \int^{2\eta}_{0} \rho^\frac52 |u_1( \rho)| d\rho \right|^2 d\eta\nonumber\\
  \les \ & \int_0^\infty \left|\mathcal{M}\big({\mathbb I}_{[0,\infty)}(\rho) \rho^{\frac52} u_1(\rho)\big)\right|^2(\eta) d\eta\nonumber\\
  \les \ &  \int_0^\infty |u_1(\rho)|^2 \rho^5 d\rho \les 1,
\end{align}
where $\mathcal{M}$ is the Hardy-Littlewood maximal function and we used the fact that $\mathcal{M}$  is bounded in $L^2$.

Next, we consider the contribution of \eqref{Duhamel} to the energy \eqref{u-energy-part1}. Also, for simplicity, we consider
\begin{align}
  \big(\pa_{\tau} + \pa_{s} \big)\eqref{Duhamel}\, =\, \,& e^{\tau+s}   \int^{e^\tau\cosh s}_{1-\delta_1} (e^{\tau+s}-t)^2  (|\bu|\bu)(t,e^{\tau+s}-t) dt\label{Duhamel+1}\\
  & + \frac{e^{\tau+s}}{(e^{\tau+s}-e^{\tau-s})^2}   \int^{e^\tau\cosh s}_{1-\delta_1}  \int^{e^{\tau+s}-t}_{t-e^{\tau-s}} \rho (t-e^{\tau-s})^2(|\bu|\bu) (t,\rho) d\rho dt\label{Duhamel+2} \\
  & -\frac{e^{\tau+s}}{(e^{\tau+s}-e^{\tau-s})^2}   \int^{e^\tau\cosh s}_{1-\delta_1}  \int^{e^{\tau+s}-t}_{t-e^{\tau-s}} \rho^3(|\bu|\bu)(t,\rho) \label{Duhamel+3}d\rho dt,
\end{align}
and
\begin{align}
 \big(\pa_{\tau} - \pa_{s} \big)\eqref{Duhamel}\, =\, \,
  & e^{\tau-s}  \int^{e^\tau\cosh s}_{1-\delta_1} (t-e^{\tau-s})^2 |\bu|\bu(t,t-e^{\tau-s}) dt\label{Duhamel-1}\\
  & - \frac{e^{\tau-s}}{(e^{\tau+s}-e^{\tau-s})^2}   \int^{e^\tau\cosh s}_{1-\delta_1}  \int^{e^{\tau+s}-t}_{t-e^{\tau-s}} \rho(e^{\tau+s}-t)^2|\bu|\bu(t,\rho) d\rho dt\label{Duhamel-2} \\
  & +\frac{e^{\tau-s}}{(e^{\tau+s}-e^{\tau-s})^2}   \int^{e^\tau\cosh s}_{1-\delta_1}  \int^{e^{\tau+s}-t}_{t-e^{\tau-s}} \rho^3|\bu|\bu(t,\rho) \label{Duhamel-3}d\rho dt.
\end{align}

By the definition of $\bu$, the inequality \eqref{u-out-cone} and H\"older, the contribution of \eqref{Duhamel+1} can be estimated as
\begin{align}
&\ \int_{s_0}^\infty \left| e^{\tau+s}\int^{e^\tau\cosh s}_{1-\delta_1}   (e^{\tau+s}-t)^2 |\bu|^2(t,e^{\tau+s}-t)  dt \right|^2 ds \nonumber \\
\les &\  \int_{s_0}^\infty   \int^{e^\tau\cosh s}_{1-\delta_1}  |\bu|^{4}(t,e^{\tau+s}-t)  e^{7s}   dtds  \nonumber\\
\les &\  \int_{\frac14{e^{s_0}} }^\infty   \int^\rho_{1-\delta_1}  |\bu|^{4}(t,\rho)  \rho^6   dtd\rho  \nonumber \\
\les &\  \int_0^\infty  \int_{\rho>t+\frac12}  |u|^{4}(t,\rho)  \rho^6   d\rho  dt \nonumber \\
\les &\  \norm{u}^3_{L^3(\{|x|>|t|+\frac12\})}  \sup_{t\geq 0}\norm{|x|^2 u(t,x)}_{L^\infty_x  (\{|x|>t+\frac12\} ) }   \les 1.
\end{align}

Now, we consider \eqref{Duhamel+2} and \eqref{Duhamel+3}.
By H\"older, a change of  variables,  
and  the inequality \eqref{u-out-cone}, we  have
\begin{align}
&\, \int_{s_0}^\infty \left| \int^{e^\tau\cosh s}_{1-\delta_1}  e^{-s}  \int^{e^{\tau+s}-t}_{t-e^{\tau-s}} ( t^2 \rho+ \rho^3 )|\bu|^2(t,\rho)  d\rho dt\right|^2 ds \nonumber\\
\les &\,  \int_{s_0}^\infty \left| \int^{e^\tau\cosh s}_{1-\delta_1}   \mathcal{M} \big({\mathbb I}_{[t-\frac12+\delta_1,\infty)}(\rho) \rho^3|\bu|^2(t,\rho)\big)(e^\tau\sinh s) dt  \right|^2 ds \nonumber\\
\les &\,  \int_{s_0}^\infty \int^{e^\tau\cosh s}_{1-\delta_1} [ \mathcal{M} \big({\mathbb I}_ {[t-\frac12+\delta_1,\infty)}(\rho) \rho^3|\bu|^2(t,\rho)\big)(e^\tau\sinh s)] ^2 e^\tau \cosh s dt ds \nonumber\\
\les  &\   \int^{\infty}_{1-\delta_1} \int_{e^{s_0}}^\infty
[ \mathcal{M} \big({\mathbb I}_{[t-\frac12+\delta_1,\infty)}(\rho) \rho^3|\bu|^2(t,\rho)\big)(r)] ^2 drdt\nonumber \\
\les &\, \int^{\infty}_{1-\delta_1} \int_{r>t-\frac12+\delta_1} r^6 |\bu|^4(t,r)    drdt  \nonumber \\
\les &\,  \int_0^\infty  \int_{\rho>t+\frac12}  |u|^{4}(t,\rho)  \rho^6   d\rho  dt \nonumber\\
\les &\,  \norm{u}^3_{L^3(\{|x|>|t|+\frac12\})}  \sup_{t\geq 0}\norm{|x|^2 u(t,x)}_{L^\infty_x  (\{|x|>t+\frac12\} ) } \les 1.
\end{align}
Thus, the contribution of \eqref{Duhamel+2} and \eqref{Duhamel+3} to \eqref{u-energy-part1} is finite.
\vskip 0.3cm

For \eqref{Duhamel-1}, by the fact that $e^{\tau-s_0} <\frac12-\delta_1$, the definition of $\bu$, and  the inequality \eqref{u-out-cone}, we have
\begin{align}
 &\,   \int_{s_0}^\infty e^{-2s} \left|\int_{1-\delta_1}^{e^\tau\cosh s} (t-e^{\tau-s})^2|\bu|^2(t,t-e^{\tau-s}) dt\right|^2 ds \nonumber\\
 \les& \, \int_{s_0}^\infty e^{-2s} \Big|\int_{1-\delta_1}^{e^\tau\cosh s} t^{-2} dt  \Big|^2 \les\,   
 1.
\end{align}
Similarly, for \eqref{Duhamel-2} and \eqref{Duhamel-3}, by the fact that  $e^{\tau-s_0} <\frac12-\delta_1$ 
 and  the inequality \eqref{u-out-cone}, we can obtain that

\beq
\label{u^2-tau-s}
\begin{split}
 & \int_{s_0}^\infty e ^{-6s} \left|\int_{1-\delta_1}^{e^\tau\cosh s}  \int^{e^{\tau+s}-t}_{t-e^{\tau-s}}    [\rho^3 + \rho(e^{\tau+s}-t)^2 ]   |\bu|^2(t,\rho)  d\rho dt      \right|^2 ds  \\
 \les & \int_{s_0}^\infty   e ^{-2s}\left|\int_{1-\delta_1}^{e^\tau\cosh s}  \int^{e^{\tau+s}-t}_{t-e^{\tau-s}}   \rho   |\bu|^2(t,\rho)  d\rho dt      \right|^2 ds\\
  \les & \int_{s_0}^\infty    e ^{-2s} \left|\int_{1-\delta_1}^{e^\tau\cosh s}   \int^{e^{\tau+s}-t}_{t-e^{\tau-s}}  \rho^{-3}  d\rho dt      \right|^2 ds\\
    \les & \int_{s_0}^\infty   e ^{-2s}\left|\int_{1-\delta_1}^{e^\tau\cosh s}   t^{-2} dt      \right|^2 ds\les 1 .
\end{split}
\eeq
Hence, combining \eqref{u-0-1-tau+s}--\eqref{u^2-tau-s}, we have

\beq\label{large-s<inf}
\int_{s_0}^\infty \frac12|(s^2\tw)_\tau|^2 +\frac12 |(s^2\tw)_s|^2 \  ds\les 1.
\eeq

\subsubsection{\textbf{The hyperbolic energy for  $0\leq s\leq s_0$}}

By the hyperbolic transform \eqref{hy-trans-w}, we have
\beq
\begin{split}
  (s^2\tw)_\tau(\tau,s)\,=&\,2e^{2\tau}\sinh^2s\
                            \bw(e^\tau \cosh s,e^\tau\sinh s)+e^{3\tau} \sinh^2 s \cosh s \ \bw_t(e^\tau \cosh s,e^\tau\sinh s)\\
                          &+\, e^{3\tau} \sinh^3 s \  \bw_r(e^\tau \cosh s,e^\tau\sinh s),
\end{split}
\eeq
\beq\begin{split}
(s^2\tw)_s(\tau,s)\,=&\,2e^{2\tau}\sinh s \cosh s\  \bw (e^\tau \cosh s,e^\tau\sinh s)+\,e^{3\tau} \sinh^3 s\ \bw _t (e^\tau \cosh s,e^\tau\sinh s)\\
                     &+\, e^{3\tau} \sinh^2 s\  \cosh s \bw _r(e^\tau \cosh s,e^\tau\sinh s).
\end{split}\eeq
Hence
\begin{align}
 &  \int_0^{1}  \int_0^{s_0} |(s^2\tw)_\tau|^2 +|(s^2\tw)_s|^2  ds d\tau \nonumber  \\ 
   \les\,\, & \int_0^{1}  \int_0^{s_0} e^{4\tau} \sinh^2s (\sinh^2 s+\cosh^2 s)|\bw|^2(e^\tau\cosh s, e^\tau\sinh s)  ds d\tau \label{small-s-1}   \\
        & +  \int_0^{1}  \int_0^{s_0} e^{6\tau} \sinh^4s(\sinh^2s+ \cosh^2 s)  [ \bw_t ^2 +  \bw_r^2] (e^\tau\cosh s, e^\tau\sinh s) ds d\tau \label{small-s-2}.
\end{align}
Taking $C_0=e^{1+s_0}$, by a change of variables, the Hardy inequality and the inequality \eqref{w-energy-2}, we obtain
\beq
\begin{split}
   \eqref{small-s-1} \,\les \, \iint_{|x|+|t|\leq C_0 }\frac1{|x|^2} |\bw|^2(t,x) dx dt\, \les\, \sup_{0<t<C_0} \norm{\nabla_x \bw}_{L^2_x(\R^5)}\,\les\, 1.
\end{split}
\eeq
Similarly, for \eqref{small-s-2}, changing of variables, we have
\beq
\begin{split}
   \eqref{small-s-2} \les \iint_{|x|+|t|\leq C_0 } |\nabla_{t,x} \bw|^2(t,x) dx dt\, \les\,
   \sup_{0<t<C_0} \norm{\nabla_{t,x} \bw}_{L^2_x(\R^5)}\,\les\, 1.
\end{split}
\eeq
Then, by the mean value theorem,  there exists $\tau_0\in[0,1]$, such that

\beq
  \int_0^{s_0} |(s^2\tw)_\tau|^2(\tau_0,s) +|(s^2\tw)_s|^2(\tau_0,s)  ds  \,\les\, 1.
\eeq
This estimate along with \eqref{large-s<inf} implies
\beq
  \int_0^\infty |(s^2\tw)_\tau|^2(\tau_0,s) +|(s^2\tw)_s|^2(\tau_0,s)  ds \,\les \,1.
\eeq

\subsection{Uniform boundedness of the hyperbolic energy of $\tw $}

We are now going to show that  $E_h (\tw)(\tau)$ is uniformly
bounded for $\tau\in\R_+$. 

A simple calculation gives
\beq
  \frac{d}{d\tau} E_h(\tw(\tau))\, =\, \int\frac{|s^2\tw|s^2\tw-|s^2\tu|s^2\tu }{\sinh^2s} s^2\tw_\tau ds.
\eeq
Utilizing  the decay property \eqref{v-dispersive} of $ v$, we have, for $\tau,s\geq 0$,
\beq\label{tv-1}
(e^\tau\cosh s-(1-\delta_1))^2 v(e^\tau\cosh s-(1-\delta_1),e^\tau \sinh s)\, \les\, \delta_1^{-\frac12}.
\eeq
The Huygens principle implies that  $v(e^\tau\cosh s-(1-\delta_1),e^\tau \sinh s)=0$ unless $1-\frac{6}{5}\delta_1\leq e^{\tau-s}\leq 1-\frac{4}{5}\delta_1.$
Thus, for $\tau,s\geq0$, we have
\beq\label{tv-2}
\frac{s^2|\tv(\tau,s)|}{\sinh^2 s}\,=\,e^{2\tau} |v(e^\tau\cosh s-(1-\delta_1),e^\tau \sinh s)|\,\les\,
\frac{e^{2\tau} {\mathbb I}_{\{s\geq 0:\ e^{\tau-s}\leq 1 -\frac45\delta_1\} } }{(e^\tau\cosh s-(1-\delta_1))^2}\,\les \,e^{-\tau}.
\eeq
Hence, by H\"older and interpolation, we have
\beq\label{E(w)'-1}
\begin{split}
\int_0^\infty \frac{1}{\sinh^2 s}   |s^2\tw_\tau| |s^2 \tv|^2 ds   &\les \, \left(\int_0^\infty |s^2\tw_\tau|^2 ds \right)^\frac12  \left(\int_0^\infty \frac{1}{\sinh^2s}|s^2\tv|^3 ds \right)^\frac12  \norm{\frac{s^2}{\sinh^2s} \tv(\tau,s)}_{L^\infty_s}^\frac12\\
  &\les\,  e^{-\tau/2} E_h(\tw(\tau))^\frac12  \left(\int_0^\infty \frac{1}{\sinh^2s}|s^2\tv|^3 ds \right)^\frac12,
\end{split}
\eeq

\beq\label{E(w)'-2}
\begin{split}
\int_0^\infty \frac{1}{\sinh^2 s}  |s^2\tw_\tau| |s^2 \tv| |s^2 \tw| ds &\,\les\, \left(\int_0^\infty |s^2\tw_\tau|^2 ds \right)^\frac12  \left(\int_0^\infty \frac{1}{\sinh^2s}| s^2 \tw|^3 ds \right)^\frac13
 \left(\int_0^\infty |s^2\tv(\tau,s)|^6 \frac1{ \sinh^8 s} ds\right)^\frac16\\
 &\,\les\,  e^{-\tau/2} E_h(\tw(\tau))^{ \frac56} \left(\int_0^\infty \frac{1}{\sinh^2s}|s^2\tv|^3 ds \right)^\frac16.
\end{split}
\eeq
Combining \eqref{E(w)'-1} with \eqref{E(w)'-2} and employing H\"older again,  we have
\beq\label{E(w)'-3}
\frac{d}{d\tau} E_h(\tw(\tau)\,\les\, e^{-\tau/2} \left[ E_h(\tw(\tau) + \int_0^\infty \frac{1}{\sinh^2s}|s^2\tv|^3 ds \right].
\eeq

On the other hand, changing of variables,  we have
\beq\label{tv-L3}
\int_0^\infty\int_0^\infty \frac{1}{\sinh^2s}|s^2\tv|^3 dsd\tau \leq \int_{\delta_1}^\infty \int_0^\infty |v(t,r)|^3r^4 drdt\,\leq\,\norm{v}^3_{L^3_{t,x}([\delta_1,\infty)\times\R^5)}\,\les\, 1.
\eeq
Hence, by Gronwall's inequality, \eqref{E(w)'-3} and \eqref{tv-L3} yield that  $E_h(\tw)(\tau)$ is uniformly bounded in $\R_+.$

\subsection{Concluding of the proof of Proposition \ref{u-3-bounded}} 
We complete the proof of Proposition \ref{u-3-bounded}
by studying the Morawetz action in hyperbolic coordinates. 

\begin{proposition}\label{morawetz}
Let $\tw$ be defined  in \eqref{hy-trans-w}, then
\beq\label{w-L3-cone}
\iint_{\Omega_2 } |w(t,r)|^3 r^4 dtdr\, = \,\int_0^\infty \int_0^\infty  \frac{|s^2 \tw|^3}{ \sinh^2 s}  ds d\tau\les 1.
\eeq
\end{proposition}
\begin{proof}

Define the Morawetz action by
\beq
M(\tau)\,=\,\int_0^\infty (s^2\tw)_\tau (s^2\tw)_s ds\,
=\,\int_0^\infty  \tw_\tau(\tw_s +\frac2s \tw) s^4 ds.
\eeq
One can easily find that $|M(\tau)|\leq E_h(\tw)(\tau)$.
By \eqref{tw-wave}, we have
\beq\label{M'(tau)}
\begin{split}
  \frac{d}{d\tau} M(\tau) =  &  -\int_0^\infty  \frac{2s^2\tw}{\sinh^2 s } (s^2\tw)_s\   ds  -\int_0^\infty  \frac{s^4}{\sinh^2 s} |\tu|\tu s^2\tw_s\  ds    \\
                         = & -2\int_0^\infty  \frac{|s^2 \tw|^2}{\sinh^2s} \frac{\cosh s}{ \sinh s} ds - \frac23 \int_0^\infty  \frac{|s^2 \tw|^3}{\sinh^2s} \frac{\cosh s}{ \sinh s} ds \\
                         & +\int_0^\infty  \frac{s^4}{\sinh^2 s } (|\tw|\tw- |\tu|\tu )s^2\tw_s\  ds.
\end{split}
\eeq
By H\"older, the estimate \eqref{tv-2}, and the fact that $E_h(\tw(\tau))$ is uniformly bounded for $\tau\geq 0$,  we have
\beq\begin{split}
     & \left|\int_0^\infty \int_0^\infty \frac{s^4}{\sinh^2 s } (|\tw|\tw- |\tu|\tu )s^2\tw_s\  ds d\tau\right|\\
\les&\, \int_0^\infty \int_0^\infty \frac{1}{\sinh^2 s}  |s^2\tw_\tau| |s^2 \tv|^2 ds d\tau+ \int_0^\infty\int_0^\infty \frac{1}{\sinh^2 s}  |s^2\tw_\tau| |s^2 \tv| |s^2 \tw| ds d\tau\\
\les&\,  \int_0^\infty e^{-\tau} E_h(\tw(\tau))^\frac12  (\int_0^\infty \frac{1}{\sinh^2s}|s^2\tv|^3 ds )^\frac12 \tau \\
    &+\, \int_0^\infty e^{-\tau} E_h(\tw(\tau))^{ \frac56} (\int_0^\infty \frac{1}{\sinh^2s}|s^2\tv|^3 ds )^\frac16 d\tau\\
\les&\,   \norm{v}^3_{L^3_{t,x}(\R\times\R^5)} \les 1.
\end{split}
\eeq
This  together with the equality \eqref{M'(tau)} and the fact $M(\tau)$ is uniformly bounded for $\tau\geq 0$, implies that
\beq
\int_0^\infty \int_0^\infty  \frac{|s^2 \tw|^3}{\sinh^2s} \frac{\cosh s}{ \sinh s} ds  d\tau \les 1.
\eeq
Thus, 
we have
\beq
\int_0^\infty \int_0^\infty  \frac{|s^2 \tw|^3}{ \sinh^2 s}  ds d\tau\les 1.
\eeq
This yields  \eqref{w-L3-cone} by the definition of $\tw$. 
\end{proof}


\textbf{Acknowledgments:}  This work is supported in part by the National Natural Science Foundation of China
 under grant No.11671047.


\end{document}